\documentclass[10pt,a4paper]{article}
\date{}
\usepackage{graphicx}
\usepackage{epstopdf}
\usepackage{epsfig}
\usepackage{fullpage} 

\usepackage[T1]{fontenc} 
\usepackage[utf8]{inputenc}

\usepackage{amsmath} 

\usepackage{amsfonts}
\usepackage{amssymb}
\usepackage{pstricks}
\usepackage{mathabx}
\usepackage{enumitem}

\newtheorem{theorem}{Theorem}[section]
\newtheorem{lemma}[theorem]{Lemma}
\newtheorem{proposition}[theorem]{Proposition}

\newtheorem{remark}[theorem]{Remark}

\newenvironment{proof}[1][Proof]{\begin{trivlist}
\item[\hskip \labelsep {\bfseries #1}]}{\end{trivlist}}
\newenvironment{definition}[1][Definition]{\begin{trivlist}
\item[\hskip \labelsep {\bfseries #1}]}{\end{trivlist}}

\newcommand{\norme}[1]{\left\Vert #1\right\Vert}
\newcommand{\modd}[1]{\vert #1\vert^2}

\newcommand{\nablag}{\nabla^{\gamma}}
\newcommand{\N}{\mathbb{N}}
\newcommand{\psia}{\psi_{(\alpha)}}
\newcommand{\nablamug}{\nabla^{\mu,\gamma}}

\newcommand{\Vu}{\overline{V}}
\newcommand{\Vd}{\underline{V}}

\newcommand{\ds}{\displaystyle}
\newcommand{\D}{\vert D^\gamma\vert}

\newcommand{\R}{\mathbb{R}}
\newcommand{\Z}{\mathbb{Z}}
\newcommand{\B}{\mathfrak{P}}
\newcommand{\rt}{\underline{\mathfrak{a} }}

\newcommand{\dt}{\partial_t}

\def \epsilon {\varepsilon}

\newcommand{\zetaa}{\zeta_{(\alpha)}}

\newcommand{\supp}{\textrm{supp}}

\newcommand{\E}{\mathcal{E}}
\renewcommand{\S}{\mathcal{S}}
\newcommand{\Lambdad}{\Lambda_\delta}

\newcommand{\G}{\mathcal{G}}

\renewcommand{\S}{\mathcal{S}}

\begin{document}

\title{A dispersive estimate for the linearized Water-Waves equations in finite depth}
\author{Mésognon-Gireau Benoît\footnote{UMR 8553 CNRS, Laboratoire de Mathématiques et Applications de l'Ecole Normale Supérieure, 75005 Paris, France. Email: benoit.mesognon-gireau@ens.fr}}
\date{}
\maketitle

\begin{abstract}
We prove a dispersive estimate for the solutions of the linearized Water-Waves equations in dimension $1$ in presence of a flat bottom. Adapting the proof from \cite{aynur} in the case of infinite depth, we prove a decay with respect to time $t$ of order $\vert t\vert^{-1/3}$ for solutions with initial data $\varphi$ such that  $\vert\varphi\vert_{H^1}$, $\vert x\partial_x\varphi\vert_{H^1}$ are bounded. We also give variants to this result with different decays for a more convenient use of the dispersive estimate. We then give an existence result for the full Water-Waves equations in weighted spaces for practical uses of the proven dispersive estimates.
\end{abstract}

\section{Introduction}

We recall here classical formulations of the Water-Waves problem. We then shortly introduce the meaningful dimensionless parameters of this problem, and then present the main results of this paper.
\subsection{Formulations of the Water-Waves problem}
The Water-Waves problem puts the motion of a fluid with a free surface into equations. We recall here two equivalent formulations of the Water Waves equations for an incompressible and irrotationnal fluid. 
\subsubsection{Free surface $d$-dimensional Euler equations} 
The motion, for an incompressible, inviscid and irrotationnal fluid occupying a domain $\Omega_t$ delimited below by a fixed bottom and above by a free surface is described by the following quantities:
\begin{itemize}[label=--,itemsep=0pt]
\item the velocity of the fluid $U=(V,w)$, where $V$ and $w$ are respectively the horizontal and vertical components;
 \item the free top surface profile $\zeta$;
 \item the pressure $P.$ 
\end{itemize}
All these functions depend on the time and space variables $t$ and $(X,z) \in\Omega_t$. There exists a function $b:\mathbb{R}^d\rightarrow \mathbb{R}$ such that the domain of the fluid at the time $t$ is given by  $$\Omega_t = \lbrace (X,z)\in\mathbb{R}^{d+1},-H_0+ b(X) < z <\zeta(t,X)\rbrace,$$ where $H_0$ is the typical depth of the water. The unknowns $(U,\zeta,P)$ are governed by the Euler equations: 
\begin{align}
\begin{cases}
\partial_t U +  U\cdot\nabla_{X,z}U = - \frac{1}{\rho}\nabla P -ge_z\text{ in } \Omega_t\\
\mbox{\rm div}(U) = 0 \text{ in } \Omega_t\\
\mbox{\rm curl}(U) = 0 \text{ in } \Omega_t .
\label{c2:euler}
\end{cases}\end{align}
\par We denote here $-ge_z$ the acceleration of gravity, where $e_z$ is the unit vector in the vertical direction, and $\rho$ the density of the fluid. Here, $\nabla_{X,z}$ denotes the $d+1$ dimensional gradient with respect to both variables $X$ and $z$. \par  \vspace{\baselineskip}
 These equations are completed by boundary conditions: 
\begin{align}
\begin{cases}
\partial_t \zeta +\underline{V}\cdot\nabla\zeta - \underline{w} = 0  \\
U\cdot n = 0 \text{ on } \lbrace z=-H_0+ b(X)\rbrace \\
P=P_{atm}  \text{ on } \lbrace z=\zeta(t,X)\rbrace, \label{c2:boundary_conditions}
\end{cases}
\end{align}
In these equations, $\underline{V}$ and $\underline{w}$ are the horizontal and vertical components of the velocity evaluated at the surface. The vector $n$ in the second equation stands for the normal upward vector at the bottom $(X,z=-H_0+b(X))$. We denote $P_{atm}$ the constant pressure of the atmosphere at the surface of the fluid. The first equation of \eqref{c2:boundary_conditions} states the assumption that the fluid particles do not cross the surface, while the second equation of \eqref{c2:boundary_conditions} states the assumption that they do not cross the bottom. The equations \eqref{c2:euler} with boundary conditions \eqref{c2:boundary_conditions} are commonly referred to as the free surface Euler equations.
\subsubsection{Craig-Sulem-Zakharov formulation} 
Since the fluid is by hypothesis irrotational, it derives from a scalar potential: $$U = \nabla_{X,z} \Phi.$$  Zakharov remarked in \cite{zakharov} that the free surface profile $\zeta$ and the potential at the surface $\psi = \Phi_{\vert z=\zeta}$ fully determine the motion of the fluid, and gave an Hamiltonian formulation of the problem. Later, Craig-Sulem, and Sulem (\cite{craigsulem1} and \cite{craigsulem2}) gave a formulation of the Water Waves equation involving the Dirichlet-Neumann operator. The following Hamiltonian system is equivalent (see \cite{david} and \cite{alazard} for more details) to the free surface Euler equations \eqref{c2:euler} and \eqref{c2:boundary_conditions}:
\begin{align}\begin{cases}
        \displaystyle{\partial_t \zeta -  \G\psi = 0} \\
        \displaystyle{\partial_t \psi + g\zeta + \frac{1}{2}\vert\nabla\psi\vert^2 - \frac{(\G\psi +\nabla\zeta\cdot\nabla\psi)^2}{2(1+\mid\nabla\zeta\mid^2)}=0, }\label{c2:ww_equation} 
\end{cases}\end{align}
where the unknowns are $\zeta$ (free top profile) and $\psi$ (velocity potential at the surface) with $t$ as time variable and  $X\in\mathbb{R}^d$ as space variable. The fixed bottom profile is $b$, and $\G$ stands for the Dirichlet-Neumann operator, that is 

\begin{equation*}
\G\psi = \G[\zeta, b]\psi = \sqrt{1+ \modd{\nabla\zeta}} \partial_n \Phi_{\vert z=\zeta},
\end{equation*} 
where $\Phi$ stands for the potential, and solves a Laplace equation with Neumann (at the bottom) and Dirichlet (at the surface) boundary conditions 
\begin{align*}\begin{cases}
\Delta_{X,z} \Phi = 0 \quad \text{in }  \lbrace (X,z)\in\R^d\times\R, -H_0+ b(X) < z < \zeta(X)\rbrace \\
\phi_{\vert z=\zeta} = \psi,\quad \partial_n \Phi_{\vert z=-H_0+ b} = 0 
\end{cases}\end{align*}
with the notation, for the normal derivative $$\partial_n \Phi_{\vert z=-H_0+b(X)} = \nabla_{X,z}\Phi(X,-H_0+b(X))\cdot n$$ where $n$ stands for the normal upward vector at the bottom $(X,-H_0+b(X))$. See also \cite{david} for more details.

%

\subsubsection{Dimensionless equations}
Since the properties of the solutions depend strongly on the characteristics of the fluid, it is more convenient to non-dimensionalize the equations by introducing some characteristic lengths of the wave motion:
\begin{enumerate}[itemsep=0pt,label=(\arabic*)]
\item The characteristic water depth $H_0$; 
\item The characteristic horizontal scale $L_x$ in the longitudinal direction; 
\item The characteristic horizontal scale $L_y$ in the transverse direction (when $d=2$);
\item The size of the free surface amplitude $a_{surf}$;
\item The size of bottom topography $a_{bott}$.
\end{enumerate}
Let us then introduce the dimensionless variables: $$x'=\frac{x}{L_x},\quad y'=\frac{y}{L_y},\quad \zeta'=\frac{\zeta}{a_{surf}},\quad z'=\frac{z}{H_0},\quad b'=\frac{b}{a_{bott}},$$ and the dimensionless variables: $$t'=\frac{t}{t_0},\quad \Phi'=\frac{\Phi}{\Phi_0},$$ where $$t_0 = \frac{L_x}{\sqrt{gH_0 }},\quad \Phi_0 = \frac{a_{surf}}{H_0}L_x \sqrt{gH_0}.$$ 

After re scaling, several dimensionless parameters appear in the equation. They are 
\begin{align*}
\frac{a_{surf}}{H_0} = \epsilon, \quad \frac{H_0^2}{L_x^2} = \mu,\quad \frac{a_{bott}}{H_0} = \beta,\quad \frac{L_x}{L_y} = \gamma,
\end{align*}
where  $\epsilon,\mu,\beta,\gamma$ are commonly referred to respectively as "nonlinearity", "shallowness", "topography" and "transversality" parameters.\par
\vspace{\baselineskip}

For instance, the Zakharov-Craig-Sulem system (\ref{c2:ww_equation}) becomes (see \cite{david} for more details) in dimensionless variables (we omit the "primes" for the sake of clarity): 
\begin{align}\begin{cases}
        \displaystyle{\partial_t \zeta - \frac{1}{\mu} \G_{\mu,\gamma}[\epsilon\zeta,\beta b]\psi = 0 }\\
        \ds \partial_t \psi + \zeta + \frac{\epsilon}{2}\vert\nablag\psi\vert^2 - \frac{\epsilon}{\mu}\frac{(\G_{\mu,\gamma}[\epsilon\zeta,\beta b]\psi +\epsilon\mu\nablag\zeta\cdot\nablag\psi)^2}{2(1+\epsilon^2\mu\mid\nablag\zeta\mid^2)}=0,  \label{c2:ww_equation1}
\end{cases}\end{align} where $\G_{\mu,\gamma}[\epsilon\zeta,\beta b]\psi$ stands for the dimensionless  Dirichlet-Neumann operator, 

\begin{equation*}
\G_{\mu,\gamma}[\epsilon\zeta,\beta b]\psi = \sqrt{1+\epsilon^2 \modd{\nablag\zeta}} \partial_n \Phi_{\vert z=\epsilon\zeta} =  (\partial_z\Phi-\mu\nabla^{\gamma}(\epsilon \zeta)\cdot\nabla^{\gamma}\Phi)_{\vert z=\epsilon\zeta},
\end{equation*} 
where $\Phi$ solves the Laplace equation with Neumann (at the bottom) and Dirichlet (at the surface) boundary conditions 
\begin{equation}\begin{aligned}
\Delta^{\mu,\gamma} \Phi = 0 \quad \text{in }  \lbrace (X,z)\in\R^d\times\R -1+\beta b(X) < z < \epsilon\zeta(X)\rbrace \\
\phi_{\vert z=\epsilon\zeta} = \psi,\quad \partial_n \Phi_{\vert z=-1+\beta b} = 0.
\end{aligned}\label{c2:dirichleta}\end{equation}
We used the following notations:
\begin{align*}
&\nabla^{\gamma} = {}^t(\partial_x,\gamma\partial_y) \quad &\text{ if } d=2\quad &\text{ and } &\nabla^{\gamma} = \partial_x &\quad \text{ if } d=1 \\
&\Delta^{\mu,\gamma} = \mu\partial_x^2+\gamma^2\mu\partial_y^2+\partial_z^2 \quad &\text{ if } d=2\quad &\text{ and } &\Delta^{\mu,\gamma} = \mu\partial_x^2+\partial_z^2 &\quad \text{ if } d=1
\end{align*}
and  $$ \partial_n \Phi_{\vert z=-1+\beta b}=\frac{1}{\sqrt{1+\beta^2\vert\nablag b\vert^2}} (\partial_z\Phi-\mu\nabla^{\gamma}(\beta b)\cdot\nabla^{\gamma}\Phi)_{\vert z=-1+\beta b}.$$ 

\subsection{Main result}
The linearized Water-Waves equations \eqref{c2:ww_equation1} in one dimension around a rest state of a flat surface and a zero velocity, in presence of a flat bottom can be read as
\begin{equation}
\left\{\begin{aligned} \partial_t\zeta -\frac{1}{\mu}\mathcal{G}_0\psi = 0\\
\partial_t \psi+\zeta=0 \\
(\zeta,\psi)(0)=(\zeta_0,\psi_0)
\end{aligned}\right.\label{c2:ww_linear}
\end{equation}
where $(t,x)\in\R\times\R$.   We denote $$\frac{1}{\mu}\mathcal{G}_0=\frac{1}{\mu}\G[0,0]$$ the Dirichlet-Neumann operator in $\zeta=0$ with a flat bottom, which explicit formulation is given by its Fourier transform \begin{equation}\frac{1}{\mu}\widehat{\mathcal{G}_0 f}(\xi) = \frac{\vert\xi\vert\tanh(\sqrt{\mu}\vert\xi\vert)}{\sqrt{\mu}}\widehat{f}(\xi) ,\label{c2:defGo}\end{equation} for all $f\in\mathcal{S}'(\R)$ where $\mu$ is the shallowness parameter (see for instance \cite{david} for more details). The equation \eqref{c2:ww_linear} leads to the following equation for $\zeta$:
$$\partial_t^2\zeta+\frac{1}{\mu}\mathcal{G}_0\zeta = 0$$ which is similar to the wave equation for low frequencies, and to the Water-Wave equation in infinite depth $$\partial_t^2 \zeta+(-\Delta)^{1/2}\zeta=0$$ where $\Delta = \partial_x^2$, for high frequencies. In order to study the solutions of the linearized system \eqref{c2:ww_linear}, we are therefore led to study the decay in time of the operator $e^{it\omega(D)}$ where \begin{equation*}\omega : \left\{ \begin{aligned}\R &\longrightarrow\R \\
\xi &\longmapsto \sqrt{\frac{\vert\xi\vert\tanh(\sqrt{\mu}\vert\xi\vert)}{\sqrt{\mu}}}.
\end{aligned}\right. \end{equation*}
The dispersive nature of the Water-Waves equations in infinite depth plays a key role in the proof of long time or global time results: see for instance \cite{wu2} for almost global well-posedness in $2d$, \cite{wu2011global} for $3d$ global well-posedness,  Ionescu-Pusateri \cite{ionescu1}, Alazard-Delort \cite{alazard_delort1} and \cite{alazard_delort2} for Global well-posedness in $2d$, \cite{germain2012global} for the global well-posedness in $3d$. However, there are to our knowledge only few results on decay estimates for the Water-Waves equations in finite depth (see for instance \cite{melinand}).  Recently, Aynur Bulut proved in \cite{aynur} an $L^2$ based norm-$L^\infty$ decay estimate for the linear Water-Waves equation in infinite depth:
\begin{equation}\vert e^{it ( -\Delta)^{1/4}}\varphi\vert_\infty \leq C(1+\vert t\vert)^{-1/2}(\vert\varphi\vert_{H^1}+\vert x\partial_x\varphi \vert_{L^2})\label{c2:aynur_estimate}.\end{equation}
As for all oscillatory integrals estimates, the proof of this result relies only on the behaviour of the operator $(-\Delta)^{1/2}$, which is the same as the behaviour of $\frac{1}{\mu}\mathcal{G}_0$ (recall the definition \eqref{c2:defGo}) for high frequencies.  We therefore adapt this proof to get a similar result in the case of a finite depth, with a very special attention given to the dependence in the shallowness parameter $\mu$. As one shall see later, this result gives Bulut's estimate in the limit $\mu$ goes to $+\infty$. We prove in Section \ref{c2:dispersion_section} of this paper the following result:
\begin{theorem}\label{c2:theorem_vague}
Let \begin{equation*}\omega : \left\{ \begin{aligned}\R &\longrightarrow\R \\
\xi &\longmapsto \sqrt{\frac{\vert\xi\vert\tanh(\sqrt{\mu}\vert\xi\vert)}{\sqrt{\mu}}}.
\end{aligned}\right. \end{equation*}
Then, there exists $C>0$ independent on $\mu$ such that, for all $\mu>0$:
$$\forall t>0,\qquad \forall \varphi\in\mathcal{S}(\R)\qquad \vert e^{it\omega(D)}\varphi\vert_\infty \leq C(\frac{1}{\mu^{1/4}}\frac{1}{(1+t/\sqrt{\mu})^{1/8}}+\frac{1}{(1+t/\sqrt{\mu})^{1/2}})(\vert\varphi\vert_{H^1}+\vert x\partial_x\varphi\vert_2).$$ 
\end{theorem}
 Though $\frac{1}{\mu}\sqrt{\G_0}$ and the square root of the wave operator $(-\Delta)^{1/2}$ have the same behaviour for low frequencies, it is not the case for the second order derivatives of these operators. For this reason, one should not be surprised to have a dispersion result for the Water-Waves equations, while the wave equation in dimension $1$ is not dispersive.  \par\vspace{\baselineskip}

The decay in $\frac{1}{t^{1/8}}$ given by Theorem \ref{c2:theorem_vague} is however very bad. As one might be interested to have a better decay result, we also prove the following result, with different spaces:
\begin{theorem}\label{c2:theorembisvague}
With the notations of Theorem \ref{c2:theorem_vague}, the following estimates hold:\begin{enumerate}[itemsep=0pt]
\item There exists $C>0$ independent on $\mu$ such that, for all $\mu>0$:
\begin{align*}\forall t>0,\quad \forall \varphi\in\mathcal{S}(\R)\quad \vert e^{it\omega(D)}\varphi\vert_\infty &\leq C(\frac{1}{\mu^{3/4}}\frac{1}{(1+t/\sqrt{\mu})^{1/3}}\vert \varphi\vert_{L^1}\\&+\frac{1}{(1+t/\sqrt{\mu})^{1/2}}(\vert\varphi\vert_{H^1}+\vert x\partial_x\varphi\vert_2)).\end{align*}
\item There exists $C>0$ independent on $\mu$ such that, for all $\mu>0$:
\begin{align*}\forall t>0,\quad \forall \varphi\in\mathcal{S}(\R)\quad \vert e^{it\omega(D)}\varphi\vert_\infty &\leq C(\frac{1}{\mu^{3/4}}\frac{1}{(1+t/\sqrt{\mu})^{1/3}}\vert x\varphi\vert_{L^2}\\&+\frac{1}{(1+t/\sqrt{\mu})^{1/2}}(\vert\varphi\vert_{H^1}+\vert x\partial_x\varphi\vert_2)).\end{align*}
\end{enumerate}
\end{theorem}
As one should remark, the decay given by Theorem \ref{c2:theorembisvague} is better than one of Theorem \ref{c2:theorem_vague}. However, for a practical use of such decay, one should prove that the solutions are bounded in $L^1$ or in $\vert x\cdot\vert_2$ norm, which is more difficult than proving a local existence result in $\vert x\partial_x\cdot\vert_2$-norm. In view of practical use of Theorem \ref{c2:theorem_vague}, we therefore prove in Section \ref{c2:sectionweight} and in dimensions $d=1,2$ a local existence result for the full Water-Waves equations \eqref{c2:ww_equation1} in weighted Sobolev spaces. The proof consists in an adaptation of the local existence result by \cite{alvarez}, and a technical proof of the commutator $[\G,x]$. \par\vspace{\baselineskip}

\begin{remark}
\begin{itemize}[label=--,itemsep=0pt]
\item All the dispersive effects proved in this paper are in dimension $d=1$. A similar result in dimension $2$ may however not be difficult to obtain, as the phase of the oscillatory integral studied has a radial symmetry.
\item As mentioned before, in all this paper, a very special attention is given to the dependence of the estimates with respect to $\mu$. It allows in the use of the dispersive estimates to identify different regimes, considering the size of the ratio $\frac{\epsilon}{\mu^{3/2}}$, in which the non-linear effects may overcome or not the linear (and thus dispersive) effects. Such study has been done for example in \cite{mesognon4}.\end{itemize}
\end{remark}

The plan of the article is the following:
\begin{itemize}[label=--,itemsep=0pt]
\item In Section \ref{c2:dispersion_section}, we prove a dispersive estimate for the linearized Water-Waves equation around a flat bottom and a flat surface in dimension $d=1$.
\item In Section \ref{c2:sectionweight}, we give a local existence result for the full Water-Waves equation \eqref{c2:ww_equation1} with non flat bottom, in weighted Sobolev spaces and in dimensions $d=1,2$.
\end{itemize}
 
 \subsection{Notations}\label{c2:notations}
We introduce here all the notations used in this paper.
 \subsubsection{Operators and quantities} Because of the use of dimensionless variables (see before the "dimensionless equations" paragraph), we use the following twisted partial operators: 
\begin{align*}
&\nabla^{\gamma} = {}^t(\partial_x,\gamma\partial_y) \quad &\text{ if } d=2\quad &\text{ and } &\nabla^{\gamma} = \partial_x &\quad \text{ if } d=1 \\
&\Delta^{\mu,\gamma} = \mu\partial_x^2+\gamma^2\mu\partial_y^2+\partial_z^2 \quad &\text{ if } d=2\quad &\text{ and } &\Delta^{\mu,\gamma} = \mu\partial_x^2+\partial_z^2 &\quad \text{ if } d=1 \\
&\nabla^{\mu,\gamma} = {}^t(\sqrt{\mu}\partial_x,\gamma\sqrt{\mu}\partial_y,\partial_z)\quad &\text{ if } d=2\quad &\text{ and } &{}^t(\sqrt{\mu}\partial_x,\partial_z) &\quad \text{ if } d=1. 
\end{align*}
\begin{remark}All the results proved in this paper do not need the assumption that the typical wave lengths are the same in both directions, i.e. $\gamma = 1$. However, if one is not interested in the dependence of $\gamma$, it is possible to take $\gamma = 1$ in all the following proofs. A typical situation where $\gamma\neq 1$ is for weakly transverse waves for which $\gamma=\sqrt{\mu}$; this leads to weakly transverse Boussinesq systems and the Kadomtsev–Petviashvili equation (see \cite{lannes_saut}).
\end{remark}
For all $\alpha = (\alpha_1,..,\alpha_d)\in\mathbb{N}^d$, we write $$\partial^\alpha = \partial^{\alpha_1}_{x_1}...\partial^{\alpha_d}_{x_d}$$ and $$\vert\alpha\vert = \alpha_1+...+\alpha_d.$$ 
We denote for all $a,b\in\R$: $$a\vee b = \max(a,b).$$ We denote, for all $\varphi\in\S'(\R^d)$, the Fourier transform of $\varphi$ by $\mathcal{F}(\varphi)$ of more simply $\widehat{\varphi}$.\par\vspace{\baselineskip}

 We use the classical Fourier multiplier 
$$\Lambda^s = (1-\Delta)^{s/2} \text{ on } \mathbb{R}^d$$ defined by its Fourier transform as $$\mathcal{F}(\Lambda^s u)(\xi) = (1+\vert\xi\vert^2)^{s/2}(\mathcal{F}u)(\xi)$$ for all $u\in\mathcal{S}'(\mathbb{R}^d)$.
The operator $\mathfrak{P}$ is defined as 
\begin{equation}
\mathfrak{P} = \frac{\vert D^{\gamma}\vert}{(1+\sqrt{\mu}\vert D^{\gamma}\vert)^{1/2}}\label{c2:defp}
\end{equation}
where $$\mathcal{F}(f(D)u)(\xi) = f(\xi)\mathcal{F}(u)(\xi)$$ is defined for any smooth function of polynomial growth $f$ and $u\in\mathcal{S}'(\mathbb{R}^d)$. The pseudo-differential operator $\mathfrak{P}$ acts as the square root of the Dirichlet Neumann operator, since $\B\sim \frac{1}{\mu}\sqrt{\G_0}$ (recall the definition of $\G_0$ given by \eqref{c2:defGo}) where the implicit constant does not depend on $\mu$. \\	
We denote as before by $\G_{\mu,\gamma}$ the Dirichlet-Neumann operator, which is defined as followed in the scaled variables:
\begin{equation*}
\G_{\mu,\gamma}\psi = \G_{\mu,\gamma}[\epsilon\zeta,\beta b]\psi = \sqrt{1+\epsilon^2 \modd{\nablag\zeta}} \partial_n \Phi_{\vert z=\epsilon\zeta} =  (\partial_z\Phi-\mu\nabla^{\gamma}(\epsilon\zeta)\cdot\nabla^{\gamma}\Phi)_{\vert z=\epsilon\zeta},
\end{equation*} 
where $\Phi$ solves the Laplace equation 

\begin{align*}
\begin{cases}
\Delta^{\gamma,\mu}\Phi = 0\\
\Phi_{\vert z=\epsilon\zeta} = \psi,\quad \partial_n \Phi_{\vert z=-1+\beta b} = 0.
\end{cases}
\end{align*}

For the sake of simplicity, we use the notation $\G[\epsilon\zeta,\beta b]\psi$ or even $\G\psi$ when no ambiguity is possible.

\subsubsection{The Dirichlet-Neumann problem}
In order to study the Dirichlet-Neumann problem \eqref{c2:dirichleta}, we need to map the domain occupied by the water $\Omega_t$ into a fixed domain (and not on a moving subset). For this purpose, we define:
$$\zeta^\delta(.,z) = \chi(\delta z\D)\zeta,\qquad b^\delta (.,z)=\chi(\delta(z+1)\D)b$$ where $\chi:\R\longrightarrow\R$ is a compactly supported smooth function equals to one in the neighbourhood of the origin, and $\delta>0$. We now introduce the following fixed strip:
$$\mathcal{S} = \mathbb{R}^d\times (-1;0)$$
and the diffeomorphism 
\begin{equation} \label{c2:diffeol}
\Sigma : \begin{aligned} \S&\rightarrow\Omega_t \\
(X,z)&\mapsto (X,(1+\epsilon\zeta^\delta(X)-\beta^\delta b(X))z+\epsilon\zeta^\delta(X))
\end{aligned}.
\end{equation}

It is quite easy to check that $\Phi$ is the variational solution of \eqref{c2:dirichleta} if and only if $\phi = \Phi\circ\Sigma$ is the variational solution of the following problem:
\begin{align}\begin{cases}
        \nabla^{\mu,\gamma}\cdot P(\Sigma)\nabla^{\mu,\gamma} \phi = 0  \label{c2:dirichletneumann}\\
        \phi_{z=0}=\psi,\quad \partial_n\phi_{z=-1} = 0,  \end{cases}
\end{align}
and where $$P(\Sigma) = \vert \det  J_{\Sigma}\vert J_{\Sigma}^{-1}~^t(J_{\Sigma}^{-1}),$$ where $J_{\Sigma}$ is the Jacobian matrix of the diffeomorphism $\Sigma$. 
\begin{remark} By smoothing the functions $\zeta$ and $b$ in the choice of the diffeomorphism $\Sigma$ as in \cite{alvarez}, we ensure a better estimate for the solutions of \eqref{c2:dirichletneumann}.
\end{remark}
For a complete statement of the result, and a proof of existence and uniqueness of solutions to these problems, see later Section \ref{c2:commut_section} and also \cite{david} Chapter 2. \par 
We introduce here the notations for the shape derivatives of the Dirichlet-Neumann operator. More precisely, we define the  open set  $\mathbf{\Gamma}\subset H^{t_0+1}(\mathbb{R}^d)^2$ as:
$$\mathbf{\Gamma} =\lbrace \Gamma=(\zeta,b)\in H^{t_0+1}(\mathbb{R}^d)^2,\quad \exists h_0>0,\forall X\in\mathbb{R}^d, \epsilon\zeta(X) +1-\beta b(X) \geq h_0\rbrace$$ and, given a $\psi\in \overset{.}H{}^{s+1/2}(\mathbb{R}^d)$, the mapping: \begin{displaymath}\G[\epsilon\cdot,\beta\cdot] : \left. \begin{array}{rcl}
&\mathbf{\Gamma} &\longrightarrow H^{s-1/2}(\mathbb{R}^d) \\
&\Gamma=(\zeta,b) &\longmapsto \G[\epsilon\zeta,\beta b]\psi.
\end{array}\right.\end{displaymath} We can prove the differentiability of this mapping. See Appendix \ref{c2:appendixA} for more details. We denote $d^j\G(h,k)\psi$ the $j$-th derivative of the mapping at $(\zeta,b)$ in the direction $(h,k)$. When we only differentiate in one direction, and when no ambiguity is possible, we simply denote $d^j\G(h)\psi$ or $d^j \G(k)\psi$.

\subsubsection{Functional spaces}
The standard scalar product on $L^2(\mathbb{R}^d)$ is denoted by $(\quad,\quad)_2$ and the associated norm $\vert\cdot\vert_2$. We will denote the norm of the Sobolev spaces $H^s(\mathbb{R}^d)$ by $\vert \cdot\vert_{H^s}$. We denote the norms of $W^{k,\infty}(\R^d)$ by $\vert\cdot\vert_{W^{k,\infty}}$ or simply $\vert\cdot\vert_\infty = \vert\cdot\vert_{W^{0,\infty}}$ when no ambiguity is possible. \par\vspace{\baselineskip}

We introduce the following functional Sobolev-type spaces, or Beppo-Levi spaces: 
\begin{definition}
We denote $\dot{H}^{s+1}(\mathbb{R}^d)$ the topological vector space 
$$\dot{H}^{s+1}(\mathbb{R}^d) = \lbrace u\in L^2_{loc}(\mathbb{R}^d),\quad \nabla u\in H^s(\mathbb{R}^d)\rbrace$$
endowed with the (semi) norm $\vert u\vert_{\dot{H}^{s+1}(\mathbb{R}^d)} = \vert\nabla u\vert_{H^s(\mathbb{R}^d)} $.
\end{definition}
Just remark that $\dot{H}^{s+1}(\mathbb{R}^d)/\mathbb{R}^d$ is a Banach space (see for instance \cite{lions}).\par

The space variables $z\in\mathbb{R}$ and $X\in\mathbb{R}^d$ play different roles in the equations since the Euler formulation (\ref{c2:euler}) is posed for $(X,z)\in \Omega_t$. Therefore, $X$ lives in the whole space $\mathbb{R}^d$ (which allows to take fractional Sobolev type norms in space), while $z$ is actually bounded.  For this reason, we denote the $L^2$ norm on $\S$ by $\norme{\cdot}$, and we introduce the following Banach spaces: 
\begin{definition}
The Banach space $(H^{s,k}((-1,0) \times\mathbb{R}^d),\Vert .\Vert_{H^{s,k}})$ is defined by 
$$H^{s,k}((-1,0) \times\mathbb{R}^d) = \bigcap_{j=0}^{k} H^j((-1,0);H^{s-j}(\mathbb{R}^d)),\quad \Vert u\Vert_{H^{s,k}} = \sum_{j=0}^k \Vert \Lambda^{s-j}\partial_z^j u\Vert_2.$$
\end{definition}
We will denote $\norme{\cdot}_{H^s}=\norme{\cdot}_{H^{s,0}}$ when no ambiguity is possible. To sum up, $\vert\cdot\vert$ will denote a norm on $\R^d$ while $\Vert\cdot\Vert$ will denote a norm on the flat strip $\S$.

\section{A dispersive estimate for the linear Water-Waves equations in dimension $1$}\label{c2:dispersion_section}
We prove in this section the dispersive estimate of Theorem \ref{c2:theorem_vague} and Theorem \ref{c2:theorembisvague}. We first introduce some classical results on the oscillatory integrals, and some technical results on the Littlewood-Paley decomposition.
\subsection{Technical tools}
\subsubsection{Littlewood-Palay decomposition}
We briefly recall the Littlewood-Paley decomposition. Let $\psi\in C_0^\infty(\R)$ be such that\begin{center}
 $\supp\psi \subset (-1;1)$, $\psi(\xi) = 1$ for $\vert \xi\vert\leq 1/2$. \end{center}  We now define for all $k\in \mathbb{Z}$, a function $\psi_k$ by: \begin{equation}
\forall\xi\in\R,\qquad \psi_k(\xi) = \psi(\xi/2^k)-\psi(\xi/2^{k-1})\label{c2:defpsik}
 \end{equation}
 which is compactly supported in $2^{k-1}\leq \vert\xi\vert\leq 2^{k+1}$. We then define the operators $P_k$ for all $k\in\Z$ by: \begin{equation}
\forall\xi\in\R,\qquad  \widehat{P_k f}(\xi) = \psi_k(\xi)\widehat{f}(\xi)\label{c2:defpk}
 \end{equation} for all $f\in\mathcal{S}'(\R)$. We recall here Bernstein's Lemma:
 \begin{lemma}Let $k\in\mathbb{Z}$ and $P_k$ defined by \eqref{c2:defpk}. For every $1\leq p\leq q\leq\infty$ and all $s\geq 0$, one has:
 $$\vert P_k g\vert_{L^q}\leq C 2^{k(1/p-1/q)} \vert P_k g\vert_{L^p}$$ 
 and
 $$\vert P_k g\vert_{L^p}\leq C 2^{-sk}\vert (-\Delta)^{s/2}P_k g\vert_{L^p}$$ for all $g\in\mathcal{S}(\R)$, where $C$ does not depend on $k, s, p, q$.\label{c2:bernstein}
 \end{lemma} 
 We also give the two following technical results (see for instance \cite{aynur} for a complete proof):
 \begin{lemma}
Let $k\in\mathbb{Z}$ and $P_k$ defined by \eqref{c2:defpk}. One has:
 $$\vert \partial_\xi \widehat{P_k\varphi}\vert_2 \leq C 2^{-k}(\vert\varphi\vert_2+\vert x\partial_x\varphi\vert_2),$$ for all $\varphi\in\mathcal{S}(\R)$, where $C$ does not depend on $k$.\label{c2:aynur1}
 \end{lemma}
\begin{lemma}
Let $k\in\mathbb{Z}$ and $P_k$ defined by \eqref{c2:defpk}. For all $s>1/2$ one has
$$\vert\widehat{P_k\varphi}\vert_\infty \leq C2^{-sk}(\vert\varphi\vert_{H^s}+\vert x\partial_x\varphi\vert_2),$$ for all $\varphi\in\mathcal{S}(\R)$, where $C$ does not depend on $k,s$.\label{c2:aynur2}
\end{lemma}

\subsubsection{Some results on oscillatory integrals }
We invoke later in this paper the following Van der Corput Lemma, which is a refinement of the stationary phase lemma:
\begin{lemma}\label{c2:vandercorput}
Let $\Phi\in C^k(\R)$, and $a<b$ be such that, either:
\begin{enumerate}[label = ( \arabic*)]
\item\quad  $\forall x\in[a;b],\qquad \vert \Phi^{(k)}(x)\vert \geq 1$ if $k>1$; 
\item\quad  $\forall x\in[a;b],\qquad\vert\Phi'(x)\vert \geq 1$ and $\Phi'$ is monotonic. 
\end{enumerate}
Then, there exists $C>0$ which only depends on $k$ such that $$\forall t>0,\qquad \vert\int_a^b e^{it\Phi(\xi)}d\xi \vert  \leq \frac{C}{t^{1/k}}.$$
\end{lemma}
Note that in the above Lemma, $C$ does not depend on $a$ nor $b$.

\subsection{Proof of the main result}
We prove in this section the following dispersive estimate for the linearized Water-Waves equations in dimension $1$:
\begin{theorem}
Let \begin{equation*}\omega : \left\{ \begin{aligned}\R &\longrightarrow\R \\
\xi &\longmapsto \sqrt{\frac{\vert\xi\vert\tanh(\sqrt{\mu}\vert\xi\vert)}{\sqrt{\mu}}}.
\end{aligned}\right. \end{equation*}
Then, there exists $C>0$ independent on $\mu$ such that, for all $\mu>0$:
$$\forall t>0,\qquad \forall \varphi\in\mathcal{S}(\R)\qquad \vert e^{it\omega(D)}\varphi\vert_\infty \leq C(\frac{1}{\mu^{1/4}}\frac{1}{(1+t/\sqrt{\mu})^{1/8}}+\frac{1}{(1+t/\sqrt{\mu})^{1/2}})(\vert\varphi\vert_{H^1}+\vert x\partial_x\varphi\vert_2).$$ \label{c2:theorem_dispersion}
\end{theorem}
Note that the dependence of $\mu$ in the dispersive estimate has been precisely mentioned. This is crucial in view of using a decay estimate for the Water-Waves equations, since the properties of the solutions, and even the dispersive properties of the problem may completely vary with respect to the size of the shallowness parameter $\mu$, as one should see by studying for example the asymptotic regimes when $\mu$ goes to zero. See for instance the Chapter 5 on shallow water models in \cite{david}. Notice that, if one sets $\lambda = \frac{t}{\sqrt{\mu}}$ in the statement of Theorem \ref{c2:theorem_dispersion}, one recovers the result by Aynur Bulut \eqref{c2:aynur_estimate} when $\mu$ goes to $+\infty$: indeed, one writes 
$e^{it\omega(D)} = e^{i\frac{t}{\sqrt{\mu}}g(\sqrt{\mu}D)}$ with $g(x)\underset{x\rightarrow+\infty}{\longrightarrow} \sqrt{\vert x\vert }$. Moreover, this result does not need any assumption on the size of $\mu$ (while $\mu\leq \mu_0$ is a common assumption in the Water-Waves results). \par\vspace{\baselineskip}

In \cite{melinand}, a similar decay as one given by Theorem \ref{c2:theorem_dispersion} is proved but only for functions $\varphi$ such that $\widehat{\varphi}(0)=0$ and with $L^1$ and $H^2$ weighted spaces, which are less convenient for practical use than $H^1$ space and $L^2$-weighted space. However, a short adaptation of the proof of \cite{melinand} shows that a decay of order $\frac{1}{t^{1/3}}$ can be obtained if $\widehat{\varphi}(0)\neq 0$, which is a better decay than one of Theorem \ref{c2:theorem_dispersion}. We can however adapt the proof of Theorem \ref{c2:theorem_dispersion}, and still get some better estimates than \cite{melinand} (without any assumption on $\widehat{\varphi}(0)$) that we also prove in this paper:
\begin{theorem}\label{c2:theorembis}
With the notations of Theorem \ref{c2:theorem_dispersion}, the following estimates hold:\begin{enumerate}[itemsep=0pt]
\item There exists $C>0$ independent on $\mu$ such that, for all $\mu>0$:
\begin{align*}\forall t>0,\quad \forall \varphi\in\mathcal{S}(\R)\quad \vert e^{it\omega(D)}\varphi\vert_\infty &\leq C(\frac{1}{\mu^{3/4}}\frac{1}{(1+t/\sqrt{\mu})^{1/3}}\vert \varphi\vert_{L^1}\\&+\frac{1}{(1+t/\sqrt{\mu})^{1/2}}(\vert\varphi\vert_{H^1}+\vert x\partial_x\varphi\vert_2)).\end{align*}
\item There exists $C>0$ independent on $\mu$ such that, for all $\mu>0$:
\begin{align*}\forall t>0,\quad \forall \varphi\in\mathcal{S}(\R)\quad \vert e^{it\omega(D)}\varphi\vert_\infty &\leq C(\frac{1}{\mu^{3/4}}\frac{1}{(1+t/\sqrt{\mu})^{1/3}}\vert x\varphi\vert_{L^2}\\&+\frac{1}{(1+t/\sqrt{\mu})^{1/2}}(\vert\varphi\vert_{H^1}+\vert x\partial_x\varphi\vert_2)).\end{align*}
\end{enumerate}

\end{theorem}
The Theorem \ref{c2:theorembis} gives a better decay than Theorem \ref{c2:theorem_dispersion} for the linear operator of the Water-Waves, however its use in view of long time results for the full Water-Waves equation \eqref{c2:ww_equation1} may require to prove that solutions are bounded in $L^1$ norm, or in $\vert x\cdot\vert_{L^2}$ norm, which is difficult. Indeed, the proof of local existence for this equation in weighted spaces requires the control of the commutator $[\G,x]$, which is difficult to get (while the one for $[\G,x]\partial_x$ is less difficult to get, see later Section \ref{c2:commut_section}).\vspace{\baselineskip}

The proof is based on a stationary phase result and a use of the Littlewood decomposition. More precisely, the control of oscillatory integrals of the form $$\int_a^b e^{it\Phi(\xi)}d\xi$$ consists in the precise study of where the phase $\Phi$ may be stationary. Here, the derivative of the phase \begin{equation}\Phi(\xi) = t(\omega(\xi)+x\xi/t)\label{c2:defPhi}\end{equation} may vanish and we are therefore led to study the behaviour of the second derivative $\Phi''$. However, $\vert\Phi''(\xi)\vert\underset{\xi\rightarrow\infty}{\sim} \frac{C}{\vert\xi\vert^{3/2}}$ and therefore $\Phi''$ cannot be bounded from below by a constant and one cannot apply Van der Corput's Lemma \ref{c2:vandercorput}. We therefore need to compensate the bad bound of $\Phi''$ with good weighted estimates on $\widehat{\varphi}$.  
\begin{proof}
The result is easy to get for $\vert t/\sqrt{\mu}\vert \leq 1$ by using the continuous injection $H^1(\R)\subset L^\infty(\R)$, and therefore we assume that $t>\sqrt{\mu}$ (the case $t<-\sqrt{\mu}$ is similar). Let fix $x\in\R$ with $x\neq 0$. In all this proof, we will denote by $C$ any constant which does not depend on $\mu, x, t,k$. As explained above, the derivative of the phase, $\Phi'$, may vanish and therefore one needs a close study of the second derivative $\Phi'' = \omega''$.  Note that $\omega(\xi) = \frac{1}{\sqrt{\mu}}g(\sqrt{\mu}\xi)$ with $g(\xi) = \sqrt{\vert\xi\vert\tanh(\vert\xi\vert)}$. It is easy to show that \begin{equation}\vert g''(\xi)\vert\underset{\xi\rightarrow 0}{\sim}\vert\xi\vert,\qquad  \vert g''(\xi)\vert\underset{\xi\rightarrow +\infty}{\sim}\vert\xi\vert^{-3/2}.\label{c2:comportementg}\end{equation} As suggested by the behaviour of $\omega''$, we split the study of the linear operator $e^{it\omega(D)}$ into low and high frequencies cases. To this purpose, one can define $\chi$ a smooth compactly supported function equals to $1$ in the neighbourhood of the origin, and write $\varphi = \chi(D)\varphi + (1-\chi(D))\varphi$, with as usual $\widehat{\chi(D)\varphi}(\xi) = \chi(\xi)\widehat{\varphi}(\xi)$ for all $\varphi\in\S(\R)$. Note that the estimate of Theorem \ref{c2:theorem_dispersion} stay true if we prove it for $\chi(D)(\sqrt\mu\D)\varphi$ or $(1-\chi(\sqrt\mu\D))\varphi$ instead of $\varphi$. \par\vspace{\baselineskip}
We write:
\begin{align}
e^{it\omega(D)}\varphi &= \int_\R e^{i(t\omega(\xi)+x\xi)} \widehat{\varphi}(\xi)d\xi \nonumber\\
&= \int_{\vert\xi\vert\leq y_0/\sqrt{\mu}}e^{i(t\omega(\xi)+x\xi)} \widehat{\varphi}(\xi)d\xi + \int_{\vert\xi\vert >y_0/\sqrt{\mu}} e^{i(t\omega(\xi)+x\xi)} \widehat{\varphi}(\xi)d\xi. \label{c2:decoupage01}
\end{align}
Taking the last remark into account, we first assume that $\widehat{\varphi}$ is compactly supported in some $[0;\frac{y_0}{\sqrt{\mu}}[$ for some $y_0>0$. In this case, we only need to  control the first term of \eqref{c2:decoupage01}. One sets $\delta>0$ and splits the integral into two parts (recall that the phase $\Phi$ is defined by \eqref{c2:defPhi}):
\begin{equation}
 \vert \int_{\vert\xi\vert\leq \frac{y_0}{\sqrt{\mu}}}e^{it\Phi(\xi)} \widehat{\varphi}(\xi)d\xi\vert \leq \vert \int_{\vert\xi\vert\leq\delta}e^{it\Phi(\xi)} \widehat{\varphi}(\xi)d\xi\vert +\vert \int_{\delta\leq \vert\xi\vert\leq \frac{y_0}{\sqrt{\mu}}}e^{it\Phi(\xi)} \widehat{\varphi}(\xi)d\xi\vert. \label{c2:split01}\end{equation}
 We now use Cauchy-Schwarz inequality to control the first integral of the right hand side of \eqref{c2:split01}:
 \begin{align}
 \vert \int_{\vert\xi\vert\leq\delta}e^{it\Phi(\xi)} \widehat{\varphi}(\xi)d\xi\vert  &\leq \sqrt{2\delta}\vert\varphi\vert_2.\label{c2:split04}
 \end{align}
For the second integral of the right hand side of \eqref{c2:split01}, we only consider the integral over $[\delta;y_0/\sqrt{\mu}]$ where $\Phi$ is smooth (the integral over $[-y_0/\sqrt{\mu};-\delta[$ is controlled by the exact same technique using the symmetry of $\Phi''$). We integrate by parts in the second integral of the right hand side of \eqref{c2:split01} (remember that $\widehat{\varphi}$ is compactly supported in $[0;\frac{y_0}{\sqrt{\mu}}[$):
 \begin{align*}
\vert \int_{\delta\leq \xi\leq \frac{y_0}{\sqrt{\mu}}}e^{it\Phi(\xi)} \widehat{\varphi}(\xi)d\xi\vert  &\leq \vert -\int_{\delta\leq \xi\leq \frac{y_0}{\sqrt{\mu}}} \int_\delta^\xi e^{it\Phi(s)}ds\frac{d}{d\xi} \widehat{\varphi}(\xi)d\xi \vert. 
\end{align*}
 Using Cauchy-Schwarz inequality, one gets:
\begin{align}
\vert \int_{\delta\leq \xi\leq \frac{y_0}{\sqrt{\mu}}}e^{it\Phi(\xi)} \widehat{\varphi}(\xi)d\xi\vert &\leq  \frac{\sqrt{y_0}}{\mu^{1/4}}\underset{\delta\leq\xi\leq\frac{y_0}{\sqrt{\mu}}}{\sup}\vert \int_\delta^\xi e^{it\Phi(s)}ds\vert \times \vert 1_{\delta\leq\vert\xi\vert}\frac{1}{\xi}(\xi\frac{d}{d\xi}\widehat{\varphi})(\xi)\vert_2 \nonumber \\
&\leq \frac{C}{\delta\mu^{1/4}} \underset{\delta\leq\xi\leq\frac{y_0}{\sqrt{\mu}}}{\sup}\vert \int_\delta^\xi e^{it\Phi(s)}ds\vert (\vert\varphi\vert_2+ \vert x\partial_x\varphi\vert_2).\label{c2:split02}
\end{align}
We now give a control of the oscillatory integral of  \eqref{c2:split02}. Recall that $\omega = \frac{1}{\sqrt{\mu}}g(\sqrt{\mu}\xi)$ with \eqref{c2:comportementg}. Therefore, there exists $C>0$ independent on $\mu$ such that $\vert\Phi''(s)\vert \geq C\mu \vert\xi\vert$ on $[0;y_0/\sqrt{\mu}]$ (see also Figure \ref{c2:figure}). Therefore, one has: $$\forall\delta\leq s\leq\frac{y_0}{\sqrt{\mu}},\qquad \vert \Phi''(s)\vert \geq \mu s\geq \mu \delta.$$ Using Van der Corput's Lemma \ref{c2:vandercorput}, one gets  the control:
\begin{equation}
\underset{\delta\leq\xi\leq\frac{y_0}{\sqrt{\mu}}}{\sup}\vert \int_\delta^\xi e^{it\Phi(s)}ds\vert \leq \frac{C}{\sqrt{\mu\delta t}}. \label{c2:bound1}
\end{equation}
Putting together \eqref{c2:bound1} into \eqref{c2:split02}, one gets:
\begin{equation}
\vert \int_{\delta\leq \xi\leq \frac{y_0}{\sqrt{\mu}}}e^{it\Phi(\xi)} \widehat{\varphi}(\xi)d\xi\vert \leq  \frac{C}{\delta\mu^{1/4}}\frac{1}{\sqrt{\mu\delta t}} (\vert\varphi\vert_2+ \vert x\partial_x\varphi\vert_2). \label{c2:split03}
\end{equation}
Combining the estimates of \eqref{c2:split04} and \eqref{c2:split03}, one obtains:
$$\vert \int_{\vert\xi\vert\leq \frac{y_0}{\sqrt{\mu}}}e^{it\Phi(\xi)} \widehat{\varphi}(\xi)d\xi\vert \leq (\sqrt{2\delta} +\frac{C}{\delta\mu^{1/4}}\frac{1}{\sqrt{\mu\delta t}} )(\vert\varphi\vert_2+ \vert x\partial_x\varphi\vert_2).$$
The above quantity is minimal for $\delta =\frac{1}{\mu^{3/8}t^{1/4}}$ and therefore, we finally get:
\begin{equation}
\vert \int_{\vert\xi\vert\leq \frac{y_0}{\sqrt{\mu}}}e^{it\Phi(\xi)} \widehat{\varphi}(\xi)d\xi\vert  \leq \frac{1}{\mu^{3/16}}{t^{1/8}}.\label{c2:conclusion1}
\end{equation}
\par\vspace{\baselineskip}

\begin{figure}[!h]
    \center
    \includegraphics[scale=.5]{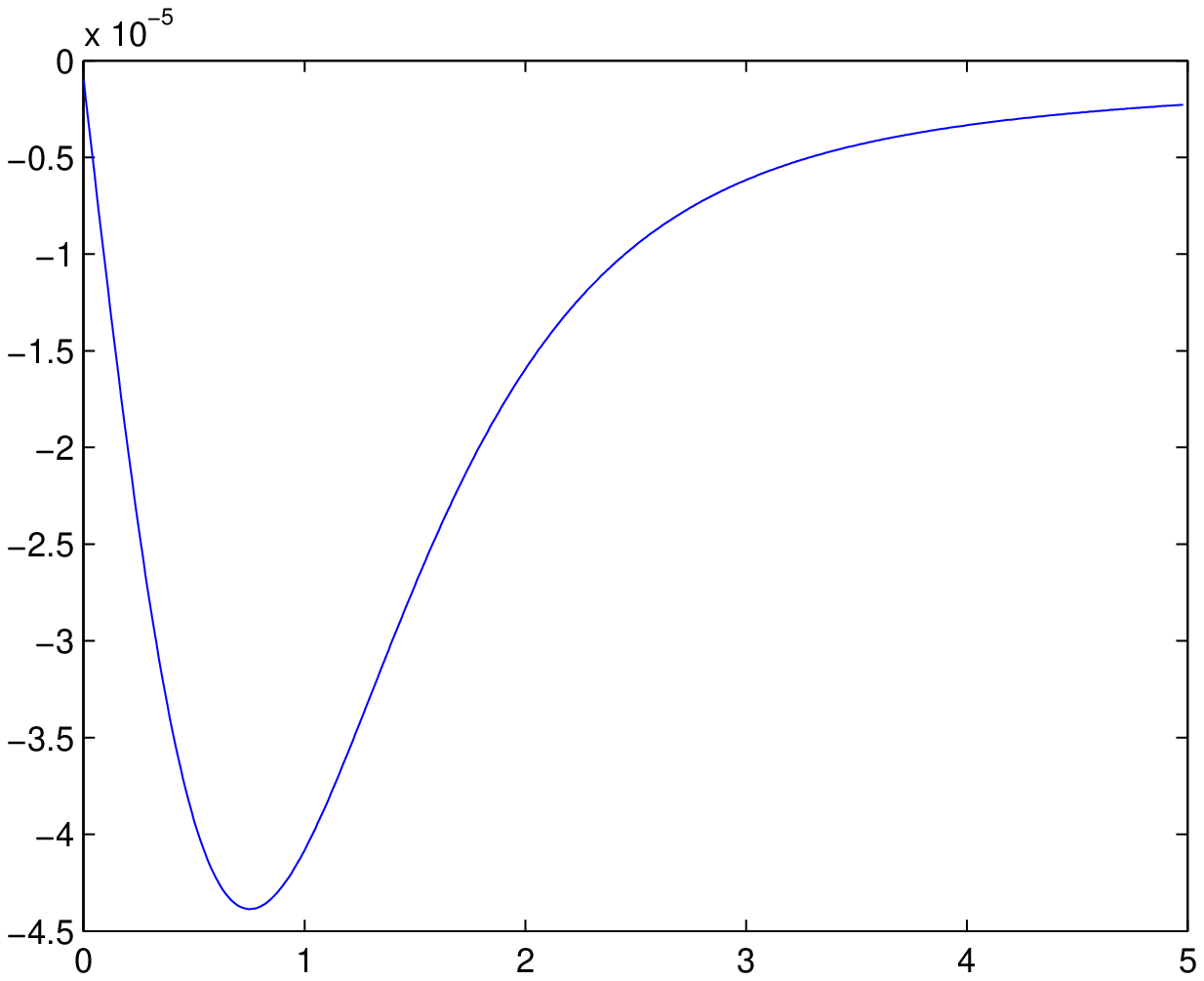}
    \caption{Graph of $\Phi''$}\label{c2:figure}
\end{figure}

We now assume that $\widehat{\varphi}$ has its support in $[\frac{y_0}{\sqrt{\mu}};+\infty[$. In this case, we only need to focus on the second term of the right hand side of \eqref{c2:decoupage01}. We are led to control in $L^\infty$ norm the quantity $\int_\R e^{it\Phi(\xi)}\widehat{\varphi}(\xi)d\xi$ where $\supp\widehat{\varphi}\subset[\frac{y_0}{\sqrt{\mu}};+\infty[$. The following lines are an adaptation of \cite{aynur}. Using the Littlewood-Paley decomposition, we split $e^{it\omega(D)}\varphi$ into
\begin{align*}
e^{it\omega(D)}\varphi &= \sum_{\substack{k\in\Z \\ 2^k\leq \lambda(t)}} P_ke^{it\omega(D)}\varphi + \sum_{\lambda(t)\leq 2^k\leq  \Lambda(t)} P_ke^{it\omega(D)}\varphi+\sum_{2^k\geq \Lambda(t)} P_ke^{it\omega(D)}\varphi
\end{align*}
where $$\lambda(t)=C_1(1+\vert t/\sqrt{\mu}\vert)^{-1},\qquad \Lambda(t) = \frac{1}{C_1}(1+\vert  t/\sqrt{\mu}\vert),$$ with $C_1>1$.  We therefore have $$\vert (e^{it\omega(D)}\varphi )(x) \vert \leq S_1+S_2+S_3,$$ with:
\begin{align*}
S_1 &= \sum_{2^k\leq \lambda(t)}\vert P_ke^{it\omega(D)}\varphi(x)\vert,\\
S_2&= \sum_{\lambda(t)\leq 2^k\leq \Lambda(t)}\vert P_ke^{it\omega(D)}\varphi(x)\vert,\\
S_3&= \sum_{2^k\geq \Lambda(t)}\vert P_ke^{it\omega(D)}\varphi(x)\vert.
\end{align*}
The term $S_1$ is controlled by using Bernstein's Lemma  \ref{c2:bernstein}:
\begin{align}
S_1&\leq \sum_{2^k\leq \lambda(t)}\vert P_ke^{it\omega(D)}\varphi(x)\vert_\infty \nonumber \\
&\leq C \sum_{2^k\leq \lambda(t)}2^{k/2}\vert P_ke^{it\omega(D)}\varphi(x) \nonumber \vert_2 \\
&\leq C \lambda(t)^{1/2}\vert \varphi(x)\vert_2 \nonumber \\
&\leq C (1+\vert  t/\sqrt{\mu}\vert)^{-1/2} \vert\varphi\vert_2. \label{c2:conclusion2}
\end{align}
Using again Bernstein's Lemma  \ref{c2:bernstein}, one gets the control of $S_3$:
\begin{align}
S_3&\leq \sum_{2^k\geq \Lambda(t)}\vert P_ke^{it\omega(D)}\varphi(x)\vert_\infty \nonumber \\
&\leq \sum_{2^k\geq \Lambda(t)}2^{k/2}\vert P_ke^{it\omega(D)}\varphi(x)\vert_2 \nonumber \\
&\leq C\sum_{2^k\geq \Lambda(t)}2^{-k/2}\vert P_ke^{it\omega(D)}\varphi(x)\vert_{H^1} \nonumber \\
&\leq C(1+\vert  t/\sqrt{\mu}\vert)^{-1/2}\vert \varphi\vert_{H^1}.\label{c2:conclusion3}
\end{align}
For the control of $S_2$, we need a close study of oscillatory integrals of the form $\int_\R e^{it(x/t\xi+\omega(\xi))}\widehat{P_k\varphi}(\xi)d\xi$. Therefore, we need precise bounds for the oscillatory phase $ \xi\mapsto (x/t) \xi+\omega(\xi)$ and its derivatives. We are led to split the summation set of $S_2$ into three parts:
\begin{align*}
I_1 &= \lbrace k\in\Z,\lambda(t)\leq 2^k\leq \Lambda(t),\quad 2^{k/2}\leq \vert t/x\vert/C_2\rbrace,\\
I_2&= \lbrace k\in\Z,\lambda(t)\leq 2^k\leq \Lambda(t),\quad \vert t/x\vert /C_2\leq 2^{k/2}\leq C_2 \vert t/x\vert\rbrace,\\
I_3&= \lbrace k\in\Z, \lambda(t)\leq 2^k\leq \Lambda(t),\quad 2^{k/2} \geq C_2 \vert t/x\vert\rbrace,
\end{align*}
where $C_2$ has to be set. We therefore set $$S_{2j} = \sum_{I_j} \vert \int_\R e^{i(x\xi+t\omega(\xi))}\widehat{P_k\varphi}(\xi)d\xi\vert,\qquad j=1,2,3.$$ The contributions of $I_1$ and $I_3$ are the most easy to get. One writes, for all $k\in I_1$:
\begin{align*}
\int_\R e^{it\Phi(\xi)}\widehat{P_k\varphi}(\xi)d\xi &= \int_\R \frac{d}{d\xi}(\int_{2^{k-1}}^\xi e^{it\Phi(s)}ds )\widehat{P_k\varphi}(\xi)d\xi \\
&= - \int_\R \int_{2^{k-1}}^\xi e^{it\Phi(s)}ds \frac{d}{d\xi}\widehat{P_k\varphi}(\xi)d\xi
\end{align*}
by integrating by parts, and recalling that $\supp \widehat{P_k\varphi}\subset \supp \psi_k\subset  \lbrace 2^{k-1}\leq\vert\xi\vert\leq 2^{k+1}\rbrace$. We now use Cauchy-Schwarz inequality to get:
\begin{equation}
\vert \int_\R e^{it\Phi(\xi)}\widehat{P_k\varphi}(\xi)d\xi\vert \leq C 2^{k/2}\underset{\xi\in \supp \psi_k\cap \supp \widehat{\varphi}}{\sup} \vert \int_{2^{k-1}}^\xi e^{it\Phi(s)}ds\vert \times \vert\frac{d}{d\xi}\widehat{P_k\varphi}\vert_2. \label{c2:decoupage03}
\end{equation}
We are therefore led to control the oscillatory integral $\int_{2^{k-1}}^\xi e^{it\Phi(s)}ds$ with $\xi\in \supp \psi_k\cap \supp \widehat{\varphi}$. We put this integral under the form $$\int_{2^{k-1}}^\xi e^{it\Phi(s)}ds = \int_{2^{k-1}}^\xi e^{it((x/t) s +\omega(s))}ds$$ and we use the following lower bound for the derivative of the phase:
$$\vert x/t+\omega'(s)\vert \geq \vert\omega'(s)\vert-\vert x/t\vert.$$ Now, recall that $k\in I_1$ and therefore $\vert x/t\vert \leq 2^{-k/2}C_2$. Moreover, recall that $\omega(\xi) = \frac{1}{\sqrt{\mu}}g(\sqrt{\mu}\xi)$ with $\vert g'(\xi)\vert\underset{\xi\rightarrow+\infty}{\sim} \vert\xi\vert^{-1/2}$, and that $\supp\widehat{\varphi}\subset [y_0/\sqrt{\mu};+\infty[$. Therefore,  the derivative of $\omega$ is bounded from below on $\supp \widehat{\varphi}$   by $\frac{C_3}{ \mu^{1/4}}\vert s\vert^{-1/2}$ where $C_3$ is independent on $\mu$. We therefore get:
$$\vert x/t+\omega'(s)\vert \geq \frac{C_3}{ \mu^{1/4}} \vert s\vert^{-1/2}-2^{-k/2}C_2.$$ Now, since $s\in \supp \psi_k$, one has 
\begin{align*}
\vert x/t+\omega'(s)\vert &\geq \frac{C_3}{ \mu^{1/4}}\frac{1}{\sqrt{2}}2^{-k/2} -2^{-k/2}C_2 \\
&\geq \frac{C_3}{ \mu^{1/4}} 2^{-k/2}
\end{align*} provided \begin{equation*}
\frac{C_3}{\mu^{1/4}\sqrt{2}}-C_2\geq C\mu^{1/4}
\end{equation*}
with $C$ independent on $k,\mu,x$. We therefore set \begin{equation}
C_2 = \frac{C_3}{2\sqrt{2}\mu^{1/4}}. \label{c2:setC2}
\end{equation} Now, one can apply the Van der Corput Lemma \ref{c2:vandercorput} and get:
\begin{equation}
 \vert\int_{2^{k-1}}^\xi e^{it((x/t) s +\omega(s))}ds \leq C\vert t\vert^{-1} 2^{k/2} \mu^{1/4}.\label{c2:decoupage04}
\end{equation}
 Using Lemma \ref{c2:aynur1}, one has:
\begin{equation}
\vert \frac{d}{d\xi}\widehat{P_k\varphi}\vert_2 \leq 2^{-k}(\vert\varphi\vert_2+\vert x\partial_x \varphi\vert_2).\label{c2:decoupage05}
\end{equation}
Putting together \eqref{c2:decoupage03}, \eqref{c2:decoupage04} and \eqref{c2:decoupage05}, one finally gets, for all $k\in I_1$:
$$\int_\R e^{it\Phi(\xi)}\widehat{P_k\varphi}(\xi)d\xi \leq \mu^{1/4}\frac{C}{\vert t\vert } (\vert\varphi\vert_2+\vert x\partial_x\varphi\vert_2).$$
We now sum over $k\in I_1$. Since the set has a $O(\log (\vert t\vert)$ number of elements (recall that it is included in $\lbrace \lambda(t)\leq 2^k\leq \Lambda(t)\rbrace$), we get:
\begin{align}
S_{21} &\leq \mu^{1/4}\frac{C}{\vert t\vert }  \sum_{k\in I_1}(\vert\varphi\vert_2+\vert x\partial_x\varphi\vert_2)\nonumber \\
&\leq \frac{C\mu^{1/4}}{\vert t\vert }\log(\vert t\vert) (\vert\varphi\vert_2+\vert x\partial_x\varphi\vert_2)\nonumber \\
&\leq C\mu^{1/4}\vert t\vert^{-1/2}(\vert\varphi\vert_2+\vert x\partial_x\varphi\vert_2)\label{c2:conclusion4}.
\end{align}
The control for $S_{23}$ is similar and therefore we omit it and focus on the most difficult term which is $S_{22}$. One starts to notice that there is a finite number of terms which is of the form $C\log(\mu)$ with $C$ independent on $t,x,\mu$ in the set $I_2$. Indeed, if $k\in I_2$ then one has $$-\log(C_2)+\log(\vert t/x\vert) \leq k \leq \log(C_2)+\log(\vert t/x\vert),$$ and $C_2$ has been set in \eqref{c2:setC2}. Therefore, it suffices to control the integrals $\int_\R e^{it\Phi(\xi)}\widehat{P_k\varphi}(\xi)d\xi$ for $k\in I_2$ by a term of the form $\frac{C}{t^{1/2}}$ with $C$ independent on $x,t,k$. For $k\in I_2$, the derivative of the phase $\Phi$ may vanishes, and one needs to control the second derivative and use a Van der Corput type result. Let $c$ be the minimum of $\vert \Phi'\vert$ on $[2^{k-1};2^{k+1}]$. \par 

\textbf{- 1st case : $\Phi'(c)=0$}
We split the integral into three terms:
\begin{equation}\int_\R e^{it\Phi(\xi)}\widehat{P_k\varphi}(\xi)d\xi = \int_{2^{k-1}}^{c-\delta} e^{it\Phi(\xi)}\widehat{P_k\varphi}(\xi)d\xi+\int_{c-\delta}^{c+\delta} e^{it\Phi(\xi)}\widehat{P_k\varphi}(\xi)d\xi+\int_{c+\delta}^{2^{k+1}} e^{it\Phi(\xi)}\widehat{P_k\varphi}(\xi)d\xi. \label{c2:decoupage06}\end{equation}
The second integral of the right hand side of \eqref{c2:decoupage06} is estimated as follows: 
\begin{align}
\vert \int_{c-\delta}^{c+\delta} e^{it\Phi(\xi)}\widehat{P_k\varphi}(\xi)d\xi \vert &\leq 2\delta \vert \widehat{P_k\varphi}(\xi) \vert_\infty\nonumber \\
&\leq 2\delta 2^{-sk} (\vert\varphi\vert_{H^s}+\vert x\partial_x\varphi\vert_2)\label{c2:decoupage10}
\end{align}
where we used Lemma \ref{c2:aynur2} to derive the last inequality, with $s>1/2$ to be set. We now focus on the control of the first integral of the right hand side of \eqref{c2:decoupage06} (the last integral is controlled by using the same technique). Integrating by parts, one gets:
\begin{align}
\vert \int_{2^{k-1}}^{c-\delta} e^{it\Phi(\xi)}\widehat{P_k\varphi}(\xi)d\xi\vert &= \vert- \int_{2^{k-1}}^{c-\delta} \int_{2^{k-1}}^{\xi} e^{it\Phi(s)}ds\frac{d}{d\xi}(\widehat{P_k\varphi})(\xi)d\xi\vert \nonumber\\
&\leq 2^{k/2}\underset{\xi\in \supp\widehat{\varphi} \cap [2^{k-1};c-\delta]}{\sup} \vert \int_{2^{k-1}}^{\xi} e^{it\Phi(s)}ds\vert  \vert \frac{d}{d\xi} \widehat{P_k\varphi}\vert_2. \label{c2:decoupage07}
\end{align}
One now estimates $\vert \int_{2^{k-1}}^{\xi} e^{it\Phi(s)}ds\vert$ for $\xi\in [2^{k-1};c-\delta]\cap \supp\widehat{\varphi}$. Recall that $\omega(\xi) = \frac{1}{\sqrt{\mu}} g(\sqrt{\mu}\xi)$ with \eqref{c2:comportementg}, and that $\supp\widehat{\varphi}\subset [y_0/\sqrt{\mu};+\infty[$. Therefore, for $\xi\in [2^{k-1};c-\delta]\cap \supp\widehat{\varphi}$, there exists $C>0$ such that $\vert \Phi''(\xi)\vert \geq \frac{C}{\mu^{1/4}\vert \xi\vert^{3/2}}$ and therefore, one has:
\begin{align*}
\vert \Phi'(s)\vert  &= \int_c^s \vert \Phi''(s)\vert ds + \Phi'(c)\\
&\geq \int_c^s \frac{C}{\mu^{1/4}\xi^{3/2}}\\
&\geq C\mu^{-1/4}(\frac{1}{\sqrt{c}}-\frac{1}{\sqrt{s}}) \\
&\geq C\mu^{-1/4}\frac{c-s}{\sqrt{c}\sqrt{s}(\sqrt{c}+\sqrt{s})}
\geq C\mu^{-1/4} \frac{\delta}{2^{3k/2}}
\end{align*}
where we used the fact that $\vert c-s\vert \geq \delta$ for $s\in [2^{k-1};\xi]$ and $\xi\in [2^{k-1};c-\delta]$. Therefore, using Van der Corput's Lemma \ref{c2:vandercorput}, one gets
\begin{equation}\int_{2^{k-1}}^{\xi} e^{it\Phi(s)}ds \leq C \frac{\mu^{1/4}2^{3k/2}}{t\delta}.\label{c2:decoupage08}\end{equation} Using Lemma \ref{c2:aynur1}, one has \begin{equation}\vert \frac{d}{d\xi} \widehat{P_k\varphi}\vert_2 \leq 2^{-k}(\vert \varphi\vert_2+\vert x\partial_x\varphi\vert_2)\label{c2:decoupage09}\end{equation} and putting together \eqref{c2:decoupage08} and \eqref{c2:decoupage09} in \eqref{c2:decoupage07}, one obtains:
\begin{equation}
\vert \int_{2^{k-1}}^{c-\delta} e^{it\Phi(\xi)}\widehat{P_k\varphi}(\xi)d\xi\vert \leq C \frac{\mu^{1/4}2^{3k/2}}{t\delta} 2^{k/2}2^{-k}(\vert \varphi\vert_2+\vert x\partial_x\varphi\vert_2). \label{c2:decoupagealpha}
\end{equation}
Putting together \eqref{c2:decoupagealpha} and \eqref{c2:decoupage10}, one finally obtains:
\begin{equation}
\vert \int_\R e^{it\Phi(\xi)}\widehat{P_k\varphi}(\xi)d\xi \vert \leq  C( \frac{\mu^{1/4}2^{k}}{t\delta} +\delta 2^{-sk})(\vert\varphi\vert_{H^1}+\vert x\partial_x\varphi\vert_2).
\end{equation}
The above right hand side is minimal with respect to $\delta$ if $\delta = 2^{sk/2} \mu^{1/8}\frac{2^{k/2}}{t^{1/2}}$. We therefore set $s=1$ and get the  control:
$$\vert \int_\R e^{it\Phi(\xi)}\widehat{P_k\varphi}(\xi)d\xi \vert \leq \mu^{1/8}\frac{C}{t^{1/2}}.$$ At last, one gets by summation on $k\in I_2$ (recall that there is a $O(\log(\mu))$ number of terms in $I_2$):

\begin{equation}S_{22} \leq \frac{\mu^{1/4}}{t^{1/2}}(\vert\varphi\vert_{H^1}+\vert x\partial_x\varphi\vert_2). \label{c2:conclusion5}\end{equation}
\textbf{-2nd case : $\Phi'(c)\neq 0$}
It is the same technique as above, noticing that in this case $\Phi'$ is monotonic and does not vanish, and therefore $c$ is one of the bounds of the integral. \par\vspace*{\baselineskip}
\textbf{Conclusion : } Putting together \eqref{c2:conclusion1},\eqref{c2:conclusion2},\eqref{c2:conclusion3},\eqref{c2:conclusion4} and \eqref{c2:conclusion5}, and taking the supremum over all $x\in\R^*$ (note that all these estimates are independent on $x$), one gets:
\begin{align*}
 \vert e^{it\omega(D)}\varphi\vert_\infty \leq C(\frac{1}{\mu^{1/4}}\frac{1}{(1+t/\sqrt{\mu})^{1/8}}+\frac{1}{(1+t/\sqrt{\mu})^{1/2}})(\vert\varphi\vert_{H^1}+\vert x\partial_x\varphi\vert_2).
\end{align*}

\qquad$\Box$
\end{proof}
We now give a short proof for a better decay estimate, but with other norms of control:

\begin{proof}[Proof of Theorem \ref{c2:theorembis}]
We use the same notations as ones of the proof of Theorem \ref{c2:theorem_dispersion}. One can control the first integral of the right hand side of \eqref{c2:decoupage01} differently:
\begin{align*}
\vert\int_{\vert\xi\vert\leq\frac{y_0}{\sqrt{\mu}}}e^{i(t\omega(\xi)+x\xi)}\varphi(\xi)d\xi\vert &\leq \vert\int_{\vert\xi\vert\leq\frac{y_0}{\sqrt{\mu}}}e^{i(t\omega(\xi)+x\xi)}d\xi\vert_\infty \vert \varphi\vert_{L^1}.
\end{align*}
We now control the integral $\int_{\vert\xi\vert\leq\frac{y_0}{\sqrt{\mu}}}e^{it\Phi(\xi)}d\xi$ in $L^\infty$ norm. One has $\Phi'''(\xi) = \omega'''(\xi) = \mu g'''(\sqrt{\mu}\xi)$ with the notations of the proof of Theorem \ref{c2:theorem_dispersion}. One can check by computation that $g'''(\xi)\underset{\xi\rightarrow 0}{\sim} -1$ and therefore, using Van der Corput's Lemma \ref{c2:vandercorput}, one gets:
$$\int_{\vert\xi\vert\leq\frac{y_0}{\sqrt{\mu}}}e^{it\Phi(\xi)}d\xi \leq \frac{C}{(\mu t)^{1/3}}.$$ One gets the first estimate of Theorem \ref{c2:theorembis}.\par\vspace{\baselineskip}

Another way to control the first integral of the right hand side of \eqref{c2:decoupage01} is by integrating by parts:
\begin{align*}
\vert\int_{\vert\xi\vert\leq\frac{y_0}{\sqrt{\mu}}}e^{i(t\omega(\xi)+x\xi)}\varphi(\xi)d\xi\vert&= \vert\int_{\vert\xi\vert\leq\frac{y_0}{\sqrt{\mu}}} \int_0^\xi e^{i(t\omega(s)+xs)}ds\frac{d}{d\xi}\widehat{\varphi}(\xi)d\xi\vert
\end{align*}
and thus, using Cauchy-Schwarz inequality, one gets:
\begin{align*}
\vert\int_{\vert\xi\vert\leq\frac{y_0}{\sqrt{\mu}}}e^{i(t\omega(\xi)+x\xi)}\varphi(\xi)d\xi\vert &\leq \frac{\sqrt{y_0}}{\mu^{1/4}}\underset{\xi\in [0;\frac{y_0}{\sqrt{\mu}}[}{\sup} \vert \int_0^\xi e^{i(t\omega(s)+xs)}ds\vert \vert\frac{d}{d\xi}\widehat{\varphi}\vert_2.
\end{align*}
One estimates as above $$\underset{\xi\in [0;\frac{y_0}{\sqrt{\mu}}[}{\sup} \vert \int_0^\xi e^{i(t\omega(s)+xs)}ds\vert \leq \frac{C}{(\mu t)^{1/3}}$$  and one gets the second estimate of Theorem \ref{c2:theorembis}.$\qquad\Box$
\end{proof}

\begin{remark} It should be possible to prove a dispersive decay in dimension $2$, using the previous results, since the phase $\Phi$ introduced in the proof of Theorem \ref{c2:theorem_dispersion} has a radial symmetry.
\end{remark}

\section{Local existence for the Water-Waves equations in weighted spaces in dimension $d=1,2$}\label{c2:sectionweight}
In view of practical use of Theorem \ref{c2:theorem_dispersion} to prove some long time (or even global time) results in the case of the full Water-Waves equations \eqref{c2:ww_equation1}, one may need to control $x\partial_x\varphi$ in $L^2$ norm, for $\varphi$ a solution of the equations.  We prove in this section a local existence result for the Water-Waves equations in weighted spaces. To this purpose, we briefly give some reminders about the Water-Waves equations, and we state the local existence result proved by \cite{alvarez}. We then give a commutator estimate which is the key point for local existence in weighted spaces. Note that in this Section we do not make any assumption on the dimension unlike in the previous one, and that all the results proved in this section stand in dimensions $d=1,2$.  The local existence result in weighted spaces is used for instance in \cite{mesognon4}, where the default of compactness in the rigid lid limit for the Water-Waves equations is investigated. \par\vspace{\baselineskip}

\subsection{The Water-Waves equations} \label{c2:recall_section} We briefly give some reminders about the Water-Waves equations and its local existence (see \cite{david} Chapter 4 for a complete study).
Let $t_0>d/2$ and $N\geq t_0+t_0\vee 2 + 3/2$ (where $a\vee b = \sup(a,b)$). The energy for the Water-Waves equations is the following (see \ref{c2:notations} for the notations):
\begin{equation}\mathcal{E}^N(U) = \vert\B\psi\vert_{H^{t_0+3/2}} + \sum_{\vert \alpha\vert\leq N} \vert\zetaa\vert_2 + \vert\psia\vert_2\label{c2:standardE}\end{equation}
where $\zetaa,\psia$ are the so called Alinhac's good unknowns:
$$\forall \alpha\in\mathbb{N}^d, \zetaa = \partial^\alpha\zeta,\qquad \psia = \partial^\alpha\psi-\epsilon\underline{w}\partial^\alpha\zeta$$
with $$\underline{w}=\frac{\G\psi+\epsilon\mu\nablag\zeta\cdot\nablag\psi}{1+\epsilon^2\mu\vert\nablag\zeta\vert^2}.$$ 
We consider solutions $U=(\zeta,\psi)$ of the Water Waves equations in the following space:
\begin{equation*}
E_{T}^N = \lbrace U\in C(\left[ 0,T\right];H^{t_0+2}\times\overset{.}H{}^2(\mathbb{R}^d)), \mathcal{E}^N(U(.))\in L^{\infty}(\left[ 0,T\right])\rbrace.
\end{equation*} The following quantity, called the Rayleigh-Taylor coefficient plays an important role in the Water-Waves problem: \begin{equation*}\rt(\zeta,\psi) = 1+\epsilon(\partial_t+\epsilon \underline{V}\cdot\nabla^{\gamma})\underline{w} = -\epsilon\frac{P_0}{\rho a g}(\partial_z P)_{\vert z = \epsilon\zeta}\label{c2:rtdef}\end{equation*} where $$\underline{V} = \nablag\psi-\epsilon\underline{w}\nablag\zeta.$$ As suggested by the notations, $\underline{V}$ and $\underline{w}$ are respectively the horizontal and vertical components of the velocity evaluated at the surface. We recall that the notation $a\vee b$ stands for $\max(a,b)$.  We can now state the local existence result by Alvarez-Samaniego Lannes (see \cite{alvarez} and \cite{david} Chapter 3 for reference):
\begin{theorem}\label{c2:uniform_result}
Let $t_0>d/2$,$N\geq t_0+t_0\vee 2+3/2$. Let $U^0 = (\zeta^0,\psi^0)\in E_0^N, b\in H^{N+1\vee t_0+1}(\R^d)$. Let $\epsilon,\beta,\gamma,\mu$ be such that $$0\leq \epsilon,\beta,\gamma\leq 1,\qquad 0\leq\mu\leq\mu_{\max}$$ with $\mu_{\max}>0$ and moreover assume that: \begin{equation*}\exists h_{min}>0,\exists a_0>0,\qquad 1+\epsilon\zeta^0-\beta b\geq h_{min}\quad\text{ and } \quad  \rt(U^0)\geq a_0.\label{c2:rayleigh}\end{equation*} Then, there exists $T>0$ and a unique solution $U^\epsilon\in E_{\frac{T}{\epsilon\vee\beta}}^N$ to \eqref{c2:ww_equation1} with initial data $U^0$. Moreover, $$\frac{1}{T}= C_1,\quad\text{ and }\quad \underset{t\in [0;\frac{T}{\epsilon\vee\beta}]}{\sup} \mathcal{E}^N(U^\epsilon(t)) = C_2$$ with $\ds C_i=C(\mathcal{E}^N(U^0),\frac{1}{h_{min}},\frac{1}{a_0})$ for $i=1,2$. \end{theorem}

\subsection{A commutator estimate}\label{c2:commut_section}
The key point of the local existence result we prove in this Section is the commutator result of Proposition \ref{c2:commut_x} below. We first need to introduce some technical results about the resolution of the Dirichlet-Neumann problem \eqref{c2:dirichleta}. We use here the notations of Section \ref{c2:recall_section}.  We recall the introduction of the diffeomorphism $\Sigma$ defined by \eqref{c2:diffeol} which maps $\Omega$ into $\mathcal{S} = \R^d\times (-1,0)$. We also recall that $\Phi$ is a solution of \eqref{c2:dirichleta} if and only if $\phi = \Phi\circ\Sigma$ is a solution of the following problem:
\begin{align}\begin{cases}
        \nabla^{\mu,\gamma}\cdot P(\Sigma)\nabla^{\mu,\gamma} \phi = 0  \label{c2:dirichletneumann2}\\
        \phi_{z=0}=\psi,\quad \partial_n\phi_{z=-1} = 0,  \end{cases}
\end{align}
where \begin{equation}P(\Sigma) = \vert \det  J_{\Sigma}\vert J_{\Sigma}^{-1}~^t(J_{\Sigma}^{-1}),\label{c2:defP}\end{equation} where $J_{\Sigma}$ is the Jacobian matrix of the diffeomorphism $\Sigma$. For the sake of clarity in the proof of the main result of this section, we introduce the following notations, for all $\zeta, b,t_0$ satisfying the hypothesis of Theorem \ref{c2:uniform_result}:
\begin{equation}
\begin{aligned}
M_0 &= C(\frac{1}{h_{\min}},\mu_{\max},\vert\zeta\vert_{H^{t_0+1}},\vert b\vert_{H^{t_0+1}}), \\
M &= C(\frac{1}{h_{\min}},\mu_{\max},\vert\zeta\vert_{H^{t_0+2}},\vert b\vert_{H^{t_0+2}}), \\
M(s)&= M_0 = C(M_0,\vert\zeta\vert_{H^{s}},\vert b\vert_{H^{s}}),
\end{aligned} \label{c2:notationM}
\end{equation}
where $C$ denotes a non decreasing function of its arguments. One can prove (see \cite{david} Chapter 2 and equation (2.26)) that \begin{equation}
P(\Sigma) = I_d + Q(\Sigma) \label{c2:decompP}
\end{equation}
with \begin{equation}\label{c2:reguQ}
\norme{Q(\Sigma)}_{H^{s,1}} \leq M_0\vert(\epsilon\zeta,\beta b)\vert_{H^{s+1/2}}.
\end{equation}
One can also prove the coercivity of $P(\Sigma)$:
\begin{equation}\forall \Theta\in\R^d,\qquad P(\Sigma)\Theta\cdot\Theta\geq k(\Sigma)\Theta^2\label{c2:pcoer}\end{equation} where $$\frac{1}{k(\Sigma)}\leq M_0.$$

\subsubsection{Technical results about the boundary problem \eqref{c2:dirichletneumann2}}   As usual for an elliptic problem of the form \eqref{c2:dirichletneumann2} with a Dirichlet condition, we are looking for solutions in the space $\psi+H_{0,surf}^1(\mathcal{S})$ where $H_{0,surf}^1(\mathcal{S})$ is the set of functions of $H^1(\mathcal{S})$ with a vanishing trace at $z=0$ (recall that $\mathcal{S}$ is the flat strip $\R^d\times (-1;0)$). More precisely, we have:
\begin{definition}We define  $H_{0,surf}^1$ as the completion of  $D(\R^d\times [-1;0[)$ endowed with the $H^1$ norm of $\mathcal{S}$.
\end{definition}
We now define the variational solutions to the elliptic equation \eqref{c2:dirichletneumann2}. To this purpose, we introduce for all $\psi\in \mathcal{S}'(\R^d)$ the smoothed distribution \begin{equation}\psi^\dagger(.,z) = \chi(\sqrt{\mu}z\D)\psi\label{c2:defdagger}\end{equation} where $\chi$ is a smooth compactly supported function equals to $1$ in the neighbourhood of the origin. \begin{definition}\label{c2:defvarsol} For all $\psi\in \dot{H}^{1/2}$ a variational solution to \eqref{c2:dirichletneumann2} is $\phi = \tilde{\phi}+\psi^\dagger$ such that $$\int_{\mathcal{S}}\nablamug \tilde{\phi}\cdot P(\Sigma)\nablamug \varphi = -\int_{\mathcal{S}} \nablamug \psi^\dagger \cdot P(\Sigma)\nablamug \varphi$$ for all $\varphi\in H_{0,surf}^1(\mathcal{S})$.\end{definition}
\begin{remark}
As expected for an elliptic problem of the form \eqref{c2:dirichletneumann2}, if $\psi\in H^{s}(\R^d)$, the solution $\Phi$ should be in $H^{s+1/2}(\mathcal{S})$ (as $\psi$ is the trace of $\Phi$ on $z=0$). Therefore, if one uses $\psi$ instead of $\psi^\dagger$ in Definition \ref{c2:defvarsol}, then the formulation provides the same regularity for $\Phi$ as for $\psi$. Instead of brutally considering $\psi$ (which is a function defined on $\R^d$) as a function of $\mathcal{S}$, we introduce $\psi^\dagger$ which is indeed $1/2$ more regular than $\psi$ and is defined on all $\mathcal{S}$. 
\end{remark}
We are now able to give the existence result for the problem \eqref{c2:dirichletneumann2} (recall the notations of \eqref{c2:notationM} for the constants $M$ and $M(s)$, and see \cite{david} for reference):
\begin{proposition} 
Let $t_0>d/2$ and $s\geq 0$. Let $\zeta,b\in H^{s+1/2}\cap H^{t_0+1}(\R^d)$ be such that $$\exists h_{\min}>0,\qquad \forall X\in\R^d,\qquad 1+\epsilon\zeta(X)-\beta b(X)\geq h_{\min}.$$ Then, for all $\psi\in \dot{H}^{s+1/2}(\R^d)$, there exists a unique variational solution $\phi$ to \eqref{c2:dirichletneumann2}. Moreover, this solution satisfies:
\begin{align*}
\forall 0\leq s\leq t_0+3/2,\qquad &\norme{\Lambda^s\nablamug\phi}_2\leq\sqrt{\mu}M(s+1/2)\vert\B\psi\vert_{H^s},\\
\forall t_0+3/2\leq s,\qquad  &\norme{\Lambda^s\nablamug\phi}_2\leq\sqrt{\mu}M(\vert\B\psi\vert_{H^s}+\vert(\epsilon\zeta,\beta b)\vert_{H^{s+1/2}}\vert\B\psi\vert_{H^{t_0+3/2}}).
\end{align*}
Moreover, if $s\geq -t_0+1$, the same estimates hold on $\norme{\nablamug\phi}_{H^{s,1}}$.
\label{c2:236}
\end{proposition}

\subsubsection{Main result}
We prove in this section an estimate for the commutator $[\frac{1}{\mu}\G,x]\partial_x$ in $H^s$ norm, where $x$ is one of the variable of $\R^d$ (we denote $x$ instead of $x_j$ for the sake of clarity). In the case of a flat bottom and a flat surface in dimension $1$, one has for all $\varphi\in\mathcal{S}(\R^d)$  and all $\xi>0$:  $$\widehat{[\frac{1}{\mu}\G,x]\partial_x\varphi} = \frac{d}{d\xi} (\frac{\tanh(\sqrt{\mu}\xi)}{\sqrt{\mu}}\xi) \xi\widehat{\varphi}(\xi) = \big( \frac{\tanh(\sqrt{\mu}\xi)}{\sqrt{\mu}}+(1-\tanh(\sqrt{\mu}\xi)^2)\xi\big)\xi \widehat{\varphi}(\xi),$$
and thus one should expect a control of the form $$ \vert [\frac{1}{\mu}G,x]\partial_x\varphi\vert_{H^{s+1/2}} \leq C \vert \B\varphi\vert_{H^{s+3/2}},$$ with $C$ independent on $\mu$, where we recall that $\B$ acts like the square root of the Dirichlet-Neumann operator and is defined by 
$$\B = \frac{\D}{(1+\sqrt{\mu}\D)^{1/2}}.$$ \begin{remark}\begin{itemize}[label=--,itemsep=0pt] \item The operator $\frac{1}{\mu}G$ has to be seen as a $3/2$ order operator instead of an one order operator if one needs a  bound which is not singular with respect to $\mu$. Indeed, the brutal bound $\frac{\tanh(\sqrt{\mu}\xi)}{\sqrt{\mu}}\leq \frac{1}{\mu^{1/2}}$ is singular in $\mu$. We still gain one derivative in the commutator $[\frac{1}{\mu}\G,x]$ and have controls uniforms with respect to $\mu$.
\item In the statement of Proposition \ref{c2:236}, we distinguish the cases $s+1/2>t_0$ and $s+1/2<t_0$. This is to have tame estimates with respect to the $H^s$ norms of the unknowns, for high values of $s$.    \end{itemize}\end{remark} 
 The following Proposition shows that the result stands true with non flat bottom and surface, and in all dimensions. We denote $\langle a\rangle_{cond} = a$ is $cond$ is satisfied, and else $0$.
\begin{proposition}\label{c2:commut_x}
Let $t_0>d/2$, $s\geq -1/2$, and $\zeta,b\in H^{t_0+2}(\R^d)$ be such that $$\exists h_{\min}>0,\quad\forall X\in\R^d,\quad 1+\epsilon\zeta(X)-\beta b(X) \geq h_{\min}.$$ We denote $x$ one of the variables of $\R^d$. Then, one has for all $\varphi\in \dot{H}^{s+2}(\R^d)$: $$\vert [\frac{1}{\mu}\G,x]\partial_x \varphi\vert_{H^{s+1/2}} \leq \mu M(s+1)\vert\B \varphi\vert_{H^{s+3/2}}+\langle \vert\B \varphi\vert_{H^{t_0+2}}\vert(\epsilon\zeta,\beta b)\vert_{H^{s+1}})\rangle_{s+1/2>t_0}.$$ Moreover, if $\nabla f\in H^{t_0+1/2}\cap H^{s+1/2}$, one has for all $\varphi\in \dot{H}^{s+2}(\R^d)$: \begin{equation*}
\begin{aligned}
\vert [\frac{1}{\mu}\G,x]f\partial_x \varphi\vert_{H^{s+1/2}} &\leq \mu M(s+1)\vert\nablag f\vert_{H^{t_0+1/2}}\vert\B \varphi\vert_{H^{s+3/2}}\\&+\langle \vert\B \varphi\vert_{H^{t_0+2}}\vert(\epsilon\zeta,\beta b)\vert_{H^{s+1}})\vert\nablag f\vert_{H^{s+1/2}}\rangle_{s+1/2>t_0}.
\end{aligned}
\end{equation*}
\end{proposition}
\begin{remark}
The fact that the commutator is applied to $\partial_x\varphi$ instead of $\varphi$ is crucial in this result. This is due to the fact that one only controls $\nablamug\Phi$ instead of $\Phi$, where $\Phi$ solves \eqref{c2:dirichletneumann2}, and thus some terms of the form $x\Phi$ are not controlled (while there derivatives are controlled). Remark that the second point of the Proposition implies the first one (just take $f=1$), but its proof requires to use the first point as one shall see during the proof below.
\end{remark}
\begin{proof}
The proof is an adaptation of the commutator estimate $[\Lambda^s,\frac{1}{\mu}G]$ which is proved in $\cite{david}$, using a duality argument. We set $$v=\partial_x\varphi.$$ For all $v\in \dot{H}^{1/2}$, we will denote $v^\mathfrak{h}$ the solution of the Dirichlet-Neumann problem \eqref{c2:dirichletneumann2} with boundary condition $v^\mathfrak{h}_{z=0} =v$. This notation stands for "harmonic extension of $v$". We recall the notation $\psi^\dagger$ for all $\psi\in\S'(\R^d)$ given by \eqref{c2:defdagger}. We now write, for all $u\in \mathcal{S}(\R^d)$:
$$(\Lambda^{s+1/2} u,[\G,x]v)_2 = (\Lambda^{s+1/2} u,\G xv)_2-(x \Lambda^{s+1/2} u,\G v)_2.$$ Since $(\Lambda^{s+1/2} x u)_{z=0}^\dagger = \Lambda^{s+1/2} xu$, we get using Green's identity that
\begin{align}
(u,[\G,x]v) &= \int_\mathcal{S} \nabla^{\mu,\gamma} \Lambda^{s+1/2} u^\dagger\cdot P(\Sigma)\nabla^{\mu,\gamma}(xv)^{\mathfrak{h}}-\int_\mathcal{S} P(\Sigma)\nabla^{\mu,\gamma}v^{\mathfrak{h}}\cdot\nablamug( x\Lambda^{s+1/2} u^\dagger)\nonumber\\
&= \int_{\mathcal{S}}\nablamug \Lambda^{s+1/2} u^\dagger\cdot P(\Sigma) \nablamug\big( (xv)^{\mathfrak{h}}-xv^{\mathfrak{h}}\big)+\int_\mathcal{S} \nablamug\Lambda^{s+1/2} u^\dagger\cdot P(\Sigma) (\nablamug x)v^\mathfrak{h} \nonumber\\
&-\int_\mathcal{S}P(\Sigma)\nablamug v^\mathfrak{h}\cdot(\nablamug x)\Lambda^{s+1/2} u^\dagger. \label{c2:decomp_commutateur}
	\end{align}
We start to control the easiest term of \eqref{c2:decomp_commutateur}, using Cauchy-Schwarz inequality (recall that $\vert\cdot\vert$ stands for norms on $\R^d$ while $\Vert\cdot\Vert$ stands for norms on the flat strip $\S = \R^d\times (-1;0)$):
\begin{align*}
\vert \int_\mathcal{S}P(\Sigma)\nablamug v^\mathfrak{h}\cdot(\nablamug x)\Lambda^{s+1/2} u^\dagger \vert &\leq \sqrt{\mu}\Vert  u^\dagger\Vert_2 \norme{\Lambda^{s+1/2} P(\Sigma)\nablamug v^\mathfrak{h}}_2
\end{align*}
where the $\sqrt{\mu}$ factor comes from the definition of $\nablamug x = ~^t(\sqrt{\mu}\partial_x,\gamma\sqrt{\mu}\partial_y,\partial_z)x$ (see Section \ref{c2:notations}) and the fact that $\partial_z x = 0$. Using the definition of $u^\dagger$, one has easily \begin{equation}
\vert u^\dagger\vert_2 \leq \Vert u\Vert_2. \label{c2:controludagger1}
\end{equation} We now use the product estimate of Proposition \ref{c2:B5} and the decomposition $P(\Sigma)=I_d+Q$ of \eqref{c2:decompP} to write:
\begin{equation}\begin{aligned}
&\forall 0\leq s+1/2\leq t_0, &&\norme{\Lambda^{s+1/2} P(\Sigma)\nablamug v^\mathfrak{h}}_2 &\leq C(1+\norme{Q}_{L^\infty H^{t_0}})\norme{\nablamug v^\mathfrak{h}}_{H^{s+1/2,0}}\\
&\forall t_0+3/2<s+1/2, &&\norme{\Lambda^{s+1} P(\Sigma)\nablamug v^\mathfrak{h}}_2 &\leq C(1+\norme{Q}_{L^\infty H^{t_0}})\norme{\nablamug v^\mathfrak{h}}_{H^{s+1/2,0}} \\ & && &+\Vert\Lambda^{s+1/2} Q\Vert_2\Vert\nablamug v^\mathfrak{h}\Vert_{L^\infty H^{t_0}}.
\end{aligned}\label{c2:decompo1}\end{equation}
\begin{remark} We don't treat the case $t_0< s \leq t_0+3/2$, since we will obtain it by interpolation of the two cases above. One has to combine the difference in the product estimate of Proposition \ref{c2:B5} between the cases $0\leq s+1\leq t_0$ and $t_0 < s$, and the difference in Lemma \ref{c2:238} between the cases $0\leq s+1\leq t_0+3/2$ and $t_0+3/2 < s$. For this reason, we split the proof in only two cases, and get the third one by interpolation.
\end{remark}
One has, using \eqref{c2:reguQ} and the embedding $H^{s+1/2,1}$ in $L^\infty H^s(\R^d)$ given by Proposition \ref{c2:212}: \begin{equation}\Vert Q\Vert_{L^\infty H^{t_0}} \leq M_0,\qquad \Vert \Lambda^{s+1/2} Q\Vert_2 \leq M_0 \vert (\epsilon\zeta,\beta b)\vert_{H^{s+1}}\label{c2:decompo2}.\end{equation} We use Proposition \ref{c2:236} and Proposition \ref{c2:212} to write:
\begin{equation}\begin{aligned}
&\forall 0\leq s+1/2\leq t_0,&&\norme{\nablamug v^\mathfrak{h}}_{H^{s+1/2,0}} &\leq \sqrt{\mu}M(s+1)\vert\B v\vert_{H^{s+1}}\\
&\forall t_0+3/2<s+1/2, &&\norme{\nablamug v^\mathfrak{h}}_{H^{s+1/2,0}} &\leq \sqrt{\mu}M(\vert\B v\vert_{H^{s+1/2}}+\vert(\epsilon\zeta,\beta b)\vert_{H^{s+1}}\vert\B v\vert_{H^{t_0+3/2}})\\
& &&\norme{\nablamug v^\mathfrak{h}}_{L^\infty H^{t_0}} &\leq \sqrt{\mu} M_0\vert\B\psi\vert_{H^{t_0+1/2}}.
\end{aligned}\label{c2:decompo3}\end{equation}
Combining \eqref{c2:decompo2} and \eqref{c2:decompo3} in \eqref{c2:decompo1}, one finally gets:
\begin{equation}\label{c2:decompo4}\begin{aligned}
&\forall 0\leq s+1/2\leq t_0,&& \norme{\Lambda^{s+1/2} P(\Sigma)\nablamug v^\mathfrak{h}}_2 &\leq \sqrt{\mu} M(s+1)\vert\B v\vert_{H^{s+1/2}} \\
&\forall t_0+3/2<s+1/2, &&\norme{\Lambda^{s+1/2} P(\Sigma)\nablamug v^\mathfrak{h}}_2 &\leq \sqrt{\mu} M(\vert\B v\vert_{H^{s+1/2}}+\vert\B v\vert_{H^{t_0+1}}\vert(\epsilon\zeta,\beta b)\vert_{H^{s+1}}).
\end{aligned}\end{equation}
Combining \eqref{c2:decompo4} with \eqref{c2:controludagger1}, one finally gets:
\begin{equation}
\vert \int_\mathcal{S}P(\Sigma)\nablamug v^\mathfrak{h}\cdot(\nablamug x)\Lambda^{s+1/2} u^\dagger \vert \leq  \mu M(s+1)\vert\B v\vert_{H^{s+1/2}}+\langle \vert\B v\vert_{H^{t_0+1}}\vert(\epsilon\zeta,\beta b)\vert_{H^{s+1}})\rangle_{s+1/2>t_0}\vert u\vert_2
\end{equation}
where we recall the notation $\langle a\rangle_{cond} = a$ if $cond$ is satisfied, and else $0$. Note that we got the result for $t_0<s<t_0+3/2$ by interpolation. Remembering that $v=\partial_x\varphi$, one gets the control of Proposition \ref{c2:commut_x} for this term. \par\vspace{\baselineskip}

We now focus on the most difficult term of \eqref{c2:decomp_commutateur} (the last term of \eqref{c2:decomp_commutateur} is estimated by a similar technique). Using Cauchy-Schwarz inequality, one gets:
\begin{align*}\vert  \int_{\mathcal{S}}\nablamug \Lambda^{s+1/2} u^\dagger\cdot P(\Sigma)\cdot \nablamug\big( (xv)^{\mathfrak{h}}-xv^{\mathfrak{h}}\big)\vert &\leq \Vert \Lambda^{s+1} P(\Sigma)\nabla^{\mu,\gamma}\big( (xv)^{\mathfrak{h}}-xv^{\mathfrak{h}}\big)\Vert_2\Vert \Lambda^{-1/2} \nabla^{\mu,\gamma} u^\dagger\Vert_2. \end{align*} 
 Using the definition of $u^\dagger$ given by \eqref{c2:defdagger}, one has easily \begin{equation}\Vert \Lambda^{-1/2} \nabla^{\mu,\gamma} u^\dagger\Vert_2 \leq C\mu^{1/4}\vert u\vert_2.\label{c2:control_commut3}\end{equation} The product estimate of Proposition \ref{c2:B5} shows that 
\begin{align*}
&\forall 0\leq s+1\leq t_0, &&\Vert \Lambda^{s+1} P(\Sigma)\nabla^{\mu,\gamma}\big( (xv)^{\mathfrak{h}}-xv^{\mathfrak{h}}\big)\Vert_2 &\leq C(1+\Vert Q\Vert_{L^\infty H^{t_0}})\Vert \Lambda^{s+1}\nabla^{\mu,\gamma}\big( (xv)^{\mathfrak{h}}-xv^{\mathfrak{h}}\big)\Vert_2\\
&\forall t_0+3/2\leq s+1,&&\Vert \Lambda^{s+1} P(\Sigma)\nabla^{\mu,\gamma}\big( (xv)^{\mathfrak{h}}-xv^{\mathfrak{h}}\big)\Vert_2 &\leq C(1+\Vert Q\vert_{L^\infty H^{t_0}})\Vert \Lambda^{s+1}\nabla^{\mu,\gamma}\big( (xv)^{\mathfrak{h}}-xv^{\mathfrak{h}}\big)\Vert_2\\ & && &+\Vert\Lambda^{s+1} Q\Vert_2\Vert\nabla^{\mu,\gamma}\big( (xv)^{\mathfrak{h}}-xv^{\mathfrak{h}}\big)\Vert_{L^\infty H^{t_0}}.
\end{align*}
One has $\Vert Q\Vert_{L^\infty H^{t_0}} \leq M_0$ and $\Vert \Lambda^{s+1} Q\Vert_2 \leq M_0 \vert (\epsilon\zeta,\beta b)\vert_{H^{s+3/2}}$. The proof of Proposition \ref{c2:commut_x} is completed if one can prove: 
\begin{equation}\begin{aligned}
\forall s\geq -1/2, \Vert \Lambda^{s+1}\nabla^{\mu,\gamma}\big( (xv)^{\mathfrak{h}}-xv^{\mathfrak{h}}\big)\Vert_2 &\leq \mu M(s+1/2)\vert\B\varphi\vert_{H^{s+1}}\\&+\langle \vert \B\varphi\vert_{H^{t_0+2}}\vert(\epsilon\zeta,\beta b)\vert_{H^{s+1}})\rangle_{s+1>t_0}.
\end{aligned}\label{c2:result_to_prove}\end{equation}
Note that the case $t_0\leq s\leq t_0+3/2$ is obtained by interpolation. We now prove the estimate \eqref{c2:result_to_prove} for $s+1\leq t_0$. The case $s+1>t_0+3/2$ is estimated by the same technique, so we omit it for the sake of clarity. The case $ t_0 \leq s\leq t_0+3/2$ is obtained by interpolation.
 The quantity $w=(xv)^\mathfrak{h}-xv^\mathfrak{h}$ satisfies the following elliptic equation:
\begin{equation}
\begin{cases}
\nablamug\cdot(P(\Sigma)\nablamug w) = -\nablamug \cdot (P(\Sigma) \nablamug x v^\mathfrak{h}) - P(\Sigma)\nablamug x\cdot \nablamug v^\mathfrak{h}\\
w_{z=0}=0,\qquad \partial_n w_{z=-1} = 0.
\end{cases}\label{c2:elliptic_2}
\end{equation}
We now prove the following elliptic regularity type result:

\begin{lemma}\label{c2:238}Let $t_0>d/2$, $s\geq 0$ and $\Sigma$ be the diffeomorphism from $\Omega$ to $\mathcal{S}$ defined by \eqref{c2:diffeol}. Let $\zeta,b\in H^{t_0+1}\cap H^{s+1/2}(\R^d)$ be such that $$\exists h_{\min}>0,\quad\forall X\in\R^d,\quad 1+\epsilon\zeta(X)-\beta b(X) \geq h_{\min}.$$  We consider the following elliptic problem:
\begin{equation}
\begin{cases}
\nablamug\cdot(P(\Sigma)\nablamug w) = -\nablamug \cdot g +f\\
w_{z=0}=0,\qquad \partial_n w_{z=-1} = 0.
\end{cases}\label{c2:elliptic_3}
\end{equation}
Then, there exists a unique variational solution $w\in H_{0,surf}^1(\mathcal{S})$ to the boundary value problem \eqref{c2:elliptic_3}.
Moreover, one has :\begin{align*} 
&\forall 0\leq s\leq t_0+3/2, &&\sqrt{\mu}\forall\Vert \Lambda^{s}\nablamug w\Vert_2\leq  M(s+1/2)(\Vert g\Vert_{H^{s,1}}+\Vert f\Vert_{H^{s-1,0}})\\
&\forall t_0+3/2\leq s, &&\sqrt{\mu}\forall\Vert \Lambda^{s}\nablamug w\Vert_2\leq M(\Vert g\Vert_{H^{s,1}}+\Vert f\Vert_{H^{s-1,0}} \\
& &&+\vert(\epsilon\zeta,\beta b)\vert_{H^{s+1/2}}(\Vert g\Vert_{H^{t_0+1/2,1}}+\Vert f\vert_{H^{t_0-1/2,0}}).
\end{align*}
\end{lemma}

\begin{proof}
By definition, $w$ is a variational solution to \eqref{c2:elliptic_3} if for all $\theta\in H_{0,surf}^1(\mathcal{S})$, one has 
\begin{equation}\begin{aligned}
\int_{\S}\nablamug w\cdot P\nablamug\theta &=- \int_{\S} (\nablamug\cdot g)\theta +\int_{\S} f \theta \\
&= \int_{\S} g\cdot \nablamug \theta +\int_{\S} f \theta - \int_{z=-1} g\theta. 
\end{aligned}\label{c2:bvp}\end{equation}
The existence and uniqueness follows from the coercivity of $P$ given by \eqref{c2:pcoer} and the Lax-Milgram Theorem. We now introduce $\Lambdad^{s} = \Lambda^{s} \chi(\delta\Lambda)$ for $\delta >0$ and $\chi$ a smooth and compactly supported function, equals to $1$ in a neighbourhood of zero. If $w\in H_{0,surf}^1(\S)$ is the variational solution of \eqref{c2:elliptic_3}, then $(\Lambdad^{s})^2 w$ is also in $H_{0,surf}^1(\S)$, and thus, taking $\theta= (\Lambdad^{s})^2w$ in \eqref{c2:bvp} (recall that $P=I+Q$):
\begin{align*}
\int_{\S}\nablamug \Lambdad^s w\cdot P\nablamug\Lambdad^s\theta &= \int_{\S} \Lambdad^s g\cdot \nablamug \Lambdad^s w +\int_{\S}\Lambdad^s f \Lambdad^s w- \int_{z=-1} \Lambdad^s g \Lambdad^s w \\&+\int_{\S} [Q,\Lambdad^s] \nablamug w\cdot \nablamug\Lambdad^s w.
\end{align*}
We now use Cauchy-Schwarz inequality and the coercivity of $P$ (see \eqref{c2:pcoer}) to get:
\begin{equation}
\begin{aligned} k(\Sigma) \Vert \nablamug\Lambdad^s w\Vert_2^2 &\leq \Vert \Lambdad^s g\Vert_2\Vert \nablamug \Lambdad^s  w\Vert_2+\Vert \Lambdad^{s-1} f\Vert_2 \Vert \Lambdad^{s+1} w\Vert_2+\vert \Lambdad^{s+1/2} w(.,-1)\vert_2\vert \Lambdad^{s-1/2} g\vert_2 \\&+ \Vert [Q,\Lambdad^s] \nablamug w\Vert_2\Vert \nablamug\Lambdad^s w\Vert_2.
\end{aligned}\label{c2:commut01}
\end{equation}
Since $\Lambdad^{s+1} w\in H_{0,surf}^1(\S)$, one has using Poincaré's inequality (recall that $\S =\R^d\times (-1,0)$ and $H_{0,surf}^1(\S)$ is the set of $H^1$ functions of $\S$ with vanishing trace at the surface):
\begin{equation}\begin{aligned}\Vert \Lambdad^{s+1} w\Vert_2 &\leq \Vert \Lambdad^{s} w\Vert_{H^{1,0}}\\
&\leq \frac{1}{\sqrt{\mu}} \Vert \Lambdad^s \nablamug w\Vert_{L^2}, \end{aligned}\label{c2:commut02}\end{equation} where the $\frac{1}{\sqrt{\mu}}$ factor comes from the definition of $\nablamug$. Moreover, one has for all $s\in\R$, using Proposition \ref{c2:212}: $$\vert u\vert_{H^{s-1/2}} \leq \Vert u\Vert_{H^{s,1}}$$ and thus one can write:
\begin{equation}\begin{aligned}
\vert \Lambdad^{s-1/2} g(.,-1)\vert_2 &\leq \Vert  g\Vert_{H^{s,1}}.
\end{aligned}\label{c2:commut03}\end{equation}
From now, the idea of the proof is to show a commutator estimate of the form
\begin{equation}\begin{aligned}
&\forall 0\leq s\leq t_0+3/2, &&\Vert [\Lambdad^{s},Q](\nablamug w)\Vert_2&\leq  M(s+1/2)\Vert\Lambdad^{s-\epsilon}\nablamug w\Vert_2\\
&\forall t_0+3/2\leq s, &&\Vert [\Lambdad^{s},Q](\nablamug w)\Vert_2 &\leq M(\vert(\epsilon\zeta,\beta b)\vert_{H^{s+1/2}}(\Vert g\Vert_{H^{t_0+1/2,1}}\\& && &+\Vert f\Vert_{H^{t_0-1/2,0}})\Vert\Lambdad^{s-\alpha}\nablamug w\Vert_2,
\end{aligned}\label{c2:commutQ}\end{equation} for some $\alpha >0$.

Putting together \eqref{c2:commut02}, \eqref{c2:commut03} and \eqref{c2:commutQ} into \eqref{c2:commut01}, letting $\delta$ goes to zero and using a finite induction on $s$, one gets the result of Lemma \ref{c2:238}. However the commutator estimate \eqref{c2:commutQ} is technical to obtain, and therefore we omit the proof for the sake of clarity (see \cite{david} Lemma 2.38 for details).\qquad$\Box$
\end{proof}

We now go back to the proof of $\eqref{c2:result_to_prove}$. For $0\leq s+1\leq t_0$, one has, using Lemma \ref{c2:238}:
\begin{align*}
\Vert \Lambda^{s+1}\nablamug((xv)^\mathfrak{h}-xv^\mathfrak{h})\Vert_2 &\leq M(s+1/2)\frac{1}{\sqrt{\mu}}(\Vert P(\Sigma)(\nablamug x)v^\mathfrak{h} \Vert_{H^{s+1,1}}+\Vert P(\Sigma)\nablamug x\cdot\nablamug v^\mathfrak{h}\Vert_{H^{s,0}})\\
&\leq  M(s+1/2)(1+\Vert Q\Vert_{H^{t_0+1,1}})(\Vert v^\mathfrak{h}\Vert_{H^{s+1,1}}+\Vert\nablamug v^\mathfrak{h}\Vert_{H^{s,0}}
\end{align*}
where we used the product estimate of Proposition \ref{c2:B5} to derive the last inequality, and where the $\frac{1}{\sqrt{\mu}}$ factor has been canceled by $\nablamug x$ which has a $\sqrt{\mu}$ factor (recall the definition of $\nablamug$).
Using Lemma \ref{c2:236}, we get the bound 
\begin{equation}
\Vert\nablamug v^\mathfrak{h}\Vert_{H^{s,0}} \leq \sqrt{\mu} \vert\B v\vert_{H^s}. \label{c2:commut04}
\end{equation}
To control $v^\mathfrak{h}$ in $H^{s+1,1}$ norm, we recall that $v=\partial_x\varphi$ and we notice that $$(\partial_x\varphi)^{\mathfrak{h}}-\partial_x (\varphi)^\mathfrak{h}\in H_{0,surf}^1(\S)$$ and we write 
\begin{align}
\Vert v^\mathfrak{h}\Vert_{H^{s+1,1}} \leq \Vert (\partial_x\varphi)^{\mathfrak{h}}-\partial_x (\varphi)^\mathfrak{h}\Vert_{H^{s+1,1}}+\Vert \partial_x (\varphi)^\mathfrak{h}\Vert_{H^{s+1,1}}. \label{c2:varphih}
\end{align}
To control the first term of the right hand side of \eqref{c2:varphih}, we use the Poincaré's inequality on the flat strip $\S$:
\begin{align*}
\Vert (\partial_x\varphi)^{\mathfrak{h}}-\partial_x (\varphi)^\mathfrak{h}\Vert_{H^{s+1,1}} &\leq \Vert \nablamug ((\partial_x\varphi)^{\mathfrak{h}}-\partial_x (\varphi)^\mathfrak{h}))\Vert_{H^{s,0}}. 
\end{align*}
Now, if one defines $$w= (\partial_x\varphi)^{\mathfrak{h}}-\partial_x (\varphi)^\mathfrak{h}$$ then $w$ satisfies the following boundary problem:

\begin{equation}\begin{cases}  \nablamug\cdot(P(\Sigma)\nablamug w) = -\nablamug\cdot g\\
w_{z=0} = 0,\qquad \partial_n w_{z=-1} = -g\cdot e_z
\end{cases}\end{equation}
with $g = [P(\Sigma),\partial_x]\nablamug \varphi^\mathfrak{h}$, and $e_z$ the unit normal vector in the vertical direction. Adapting the proof of Lemma \ref{c2:238} (see also Lemma 2.38 in \cite{david}), one can prove 
$$\Vert\nablamug w\Vert_{H^{s,0}} \leq M(s+1/2) \Vert \nablamug \varphi^\mathfrak{h}\Vert_{H^{s}}$$ and using Proposition \ref{c2:236}, one finally gets
\begin{equation}
\Vert (\partial_x\varphi)^{\mathfrak{h}}-\partial_x (\varphi)^\mathfrak{h}\Vert_{H^{s+1,1}} \leq M(s+1/2)\vert\B\varphi\vert_{H^s}.\label{c2:commut05}
\end{equation}
To control the second term of the rhs of \eqref{c2:varphih}, one uses Proposition \ref{c2:236} again:
\begin{align}
\Vert \partial_x (\varphi)^\mathfrak{h}\Vert_{H^{s+1,1}} &\leq \frac{1}{\sqrt{\mu}}\Vert\nablamug\varphi^\mathfrak{h}\Vert_{H^{s+1,1}}\nonumber\\
&\leq M(s+1/2)\vert\B\varphi\vert_{H^{s+1}}.\label{c2:commut06}\end{align}

Putting together \eqref{c2:commut05} and \eqref{c2:commut06} into \eqref{c2:varphih}, one gets
\begin{equation}
\vert v^\mathfrak{h}\vert_{H^{s+1,1}} \leq M(s+1/2) \vert \B\varphi\vert_{H^{s+1}}.\label{c2:commut07}
\end{equation}
Putting together \eqref{c2:commut04} and \eqref{c2:commut07}, we proved:
$$\Vert \Lambda^{s+1}\nablamug((xv)^\mathfrak{h}-xv^\mathfrak{h})\Vert_2\leq M(s+1/2)\vert\B\varphi\vert_{H^{s+1}}$$ which is the desired result \eqref{c2:result_to_prove}. It concludes the proof of the first point of Proposition \ref{c2:commut_x}.\par\vspace{\baselineskip}

The proof of the second point of Proposition \ref{c2:commut_x} only requires a small adaptation of the proof above. The only technical change is the control of $\Vert v^\mathfrak{h}\Vert_{H^{s+1,1}}$. We write, with $v=f\partial_x\varphi$:
\begin{equation}
\Vert v^\mathfrak{h}\Vert_{H^{s+1,1}} \leq \Vert (f\partial_x\varphi)^\mathfrak{h}-f(\partial_x\varphi )^\mathfrak{h}\Vert_{H^{s+1,1}} + \Vert f(\partial_x \varphi)^\mathfrak{h}\Vert_{H^{s+1,1}} \label{c2:split101}
\end{equation}

The second term of the right hand side of \eqref{c2:split101} is controlled using Proposition \ref{c2:B5}, and the control of $(\partial_x \varphi)^\mathfrak{h}$ proved above. To control the first term of the right hand side of \eqref{c2:split101}, one remarks that $w = (f\partial_x\varphi)^\mathfrak{h}-f(\partial_x\varphi )^\mathfrak{h}\in H^1_{0,surf}(\S)$ solves the following boundary problem:

\begin{equation}\begin{cases}  \nablamug\cdot(P(\Sigma)\nablamug w) = -\nablamug fP(\Sigma)\nablamug (\partial_x\varphi)^\mathfrak{h}-\nablamug\cdot(P(\Sigma)(\partial_x\varphi)^\mathfrak{h}\nablamug f)\\
w_{z=0} = 0,\qquad \partial_n w_{z=-1} = -P(\Sigma)(\partial_x\varphi^\mathfrak{h})\nablamug\zeta \cdot e_z
\end{cases}\end{equation}
and we use the Poincaré's inequality on the flat strip $\S$ to control $\Vert (f\partial_x\varphi)^\mathfrak{h}-f(\partial_x\varphi )^\mathfrak{h}\Vert_{H^{s+1,1}}$ by $\Vert \nablamug w\Vert_{H^{s,0}}$, and adapt the proof of Lemma \ref{c2:238} above to get the control of this latter term.

 $\Box$
\end{proof}

\subsection{Local existence in weighted Sobolev Spaces}
We prove here an existence result for the Water-Waves equation in weighted Sobolev spaces (see also \cite{nguyen2015pseudo} for another use of weighted spaces for the Water-Waves). We recall that $x$ denotes the identity of $\R^d$, and we define, for all $N\geq 2$ the energy  $\E_x^N$  by $$\E_x^N =  \E^N(\zeta,\psi)+\sum_{\alpha\in\N^{d},1\leq \vert\alpha\vert \leq N-2} \vert x\zetaa\vert_2^2+\vert\B x\psia\vert_2^2$$ where $\E^N$ is the standard energy for the Water-Waves equations given by \eqref{c2:standardE}.
\begin{theorem}\label{c2:wlocal}
Let us consider the assumptions of Theorem \ref{c2:uniform_result}, and then consider $T>0$ and $(\zeta,\psi)$ the unique solution provided by the theorem on $[0;\frac{T}{\epsilon\vee\beta}]$ of the Water-Waves equation \eqref{c2:ww_equation1}. If $(\zeta^0,\psi^0)\in \mathcal{E}^N_x$, then one has $$(\zeta,\psi)\in L^\infty([0;\frac{T}{\epsilon\vee\beta}],\mathcal{E}^N_x),$$ with $$(\zeta,\psi)_{L^\infty([0;\frac{T}{\epsilon\vee\beta}],\mathcal{E}^N_x)} \leq C_2,$$
where $C_2$ is a constant of the form $C_2=C(\mathcal{E}^N(U^0),\frac{1}{h_{min}},\frac{1}{a_0})$ with $C$ a non decreasing continuous function of its arguments.\end{theorem}
\begin{remark}\begin{itemize}[label=--,itemsep=0pt]\item Note that there are less space derivatives for the weighted norms $\vert x\zetaa\vert_2^2+\vert\B x\psia\vert_2^2$ than for the "Sobolev" norms $\E^N$. This is due to the presence of commutators of the form $[\G,x]\psia$ in the evolution equation for $\psia$, which are of order $1$ (at least) in $\psia$.
\item Note also that we control $\B\psia,\zetaa$ only for  $\vert\alpha\vert\geq 1$. This is due to the fact that we only control terms of the form $x\partial_x \varphi$.  \end{itemize}
\end{remark}
\begin{proof}
The proof is an adaptation of the proof of the Theorem \ref{c2:uniform_result} (see \cite{david} Chapter 4 for a full proof). Therefore, we only give the main ideas and insist on the specificity of using weights. Considering the result given by Theorem \ref{c2:uniform_result}, we only need to recover estimates for weighted norms (estimates for the "classical Sobolev" norms of $\E^N$ are done in the proof of the local existence result of \cite{david}). We recall (see for instance \cite{david} Chapter 3 for reference) that one has:
\begin{equation*}
(\psi,\frac{1}{\mu}G\psi)_2 \leq M_0\vert\mathfrak{P}\psi\vert_2^2\quad\text{ and }\quad \vert\mathfrak{P}\psi\vert_2^2 \leq M_0(\psi,\frac{1}{\mu}G\psi)_2 \label{c2:equivanorme}
\end{equation*}
 Therefore, we set, for all $0\leq\vert\alpha\vert\leq N-2$: $$E^\alpha = \frac{1}{2\mu}(\G x\psia,x\psia)_2+\frac{1}{2}(x\zetaa,x\zetaa)$$ and look for a control of $E^\alpha$.  We now differentiate $E^\alpha$ with respect to time and get, using the symmetry of $\G$:
$$\frac{d}{dt}E^\alpha = (\G x\psia,\partial_t x\psia)_2+(d\G(\epsilon\partial_t\zeta)x \psia,x\psia)_2+(\partial_t x\zetaa,x\zetaa)_2+\frac{1}{2}(x\zetaa,(\dt\rt)x\zetaa)_2.$$ We now need an equation in terms of $\zetaa,\psia$. To this purpose, one computes $\partial^\alpha$ of the equations \eqref{c2:ww_equation1}. One gets in the first equation a term of the form
$$\partial^\alpha \G[\epsilon\zeta,\beta b]\psi = \G[\epsilon\zeta,\beta b]\partial^\alpha \psi + \sum_{\nu<\alpha,\delta_1+...+\delta_m+l_1+...+l_n+\nu = \alpha} d\G(\partial^{\delta_1} \epsilon\zeta,...,\partial^{\delta_m} \epsilon\zeta,\partial^{l_1} \beta b,...,\partial^{l_n}\beta b)\partial^\nu \psi$$ where $d\G$ denotes the shape derivative of $\G[\epsilon\zeta,\beta b]$ with respect to the bottom $b$ and the surface $\zeta$. We therefore obtain, after computations, a system of the form (see \cite{david} Chapter 4 for details):
\begin{equation}
\left\{\begin{aligned}\partial_t\zetaa+\epsilon \Vu\cdot\nablag\zetaa-\frac{1}{\mu}\G\psia &= R^\alpha \\
\partial_t \psia +\rt \zetaa+\epsilon\Vu\cdot\nablag\psia = S^\alpha
\end{aligned}\right.\label{c2:quasisystem}
\end{equation}
with \begin{equation}\vert x R^\alpha\vert_2 + \vert\B x S^\alpha\vert_2 \leq C(\E^N_x)\label{c2:control_reste}\end{equation} with $C$ a continuous function of its arguments.   In order to get a control of the form \eqref{c2:control_reste}, one can adapt the proof of the control for the shape derivatives of $\G$ given in Proposition 3.28 of \cite{david} (we do not detail this proof here). We therefore have, replacing $\partial_t(\zetaa,\psia)$ by their expression given by \eqref{c2:quasisystem}:

\begin{equation}\begin{aligned}
\frac{d}{dt}E^\alpha&= \frac{1}{\mu}(\G x\psia,\rt x\zetaa)_2-\frac{1}{\mu}(\G x\psia,\rt x\zetaa)_2\\&+\epsilon(x\rt\zetaa,x\Vu\cdot\nablag\zetaa)_2+\frac{\epsilon}{\mu}(\G x\psia,x\Vu\cdot\nablag\psia)_2\\&+ \frac{1}{\mu}(d\G(\epsilon\partial_t\zeta)x\psia,x\psia)_2+\frac{1}{\mu}([\G,x]\psia,x\rt\zetaa)_2.
\end{aligned}\label{c2:derivee_energiea}\end{equation}
The first two terms of \eqref{c2:derivee_energiea} are the one of order $1$ with respect to the unknowns $x\zetaa,x\psia$ but cancel one another, thanks to the symmetry of the equation.\par\vspace{\baselineskip} The two terms of the second line of \eqref{c2:derivee_energiea} are of contributions of order $0$ to the energy estimate, with respect to the unknowns, thanks to the symmetry. More precisely, one computes, integrating by parts:
\begin{align*}
(x\rt\zetaa,x\Vu\cdot\nablag\zetaa)_2 &= \sum_{j=1}^d ((x\rt\zetaa,x\Vu_j\partial_j\zetaa)_2 \\
&= - \sum_{j=1}^d (x \Vu_j\partial_j\zetaa ,x\rt\zetaa)_2 - (\partial_j(\Vu_j \rt)x\zetaa,x\zetaa)_2-((\partial_j x)\Vu_j\rt \zetaa,x\zetaa)_2
\end{align*}
and therefore one has 
\begin{equation}\label{c2:estimation01}
(x\rt\zetaa,x\Vu\cdot\nablag\zetaa)_2 = -\frac{1}{2}(\partial_j(\Vu_j \rt)x\zetaa,x\zetaa)_2-\frac{1}{2}((\partial_j x)\Vu_j\rt \zetaa,x\zetaa)_2.
\end{equation}
Using Proposition \ref{c2:314}, it is possible to prove that $\vert \Vu_j\vert_{W^{1,\infty}}+\vert\rt\vert_{W^{1,\infty}}\leq \E^N$ and therefore one gets from \eqref{c2:estimation01} the control:
\begin{equation}\begin{aligned}
\vert(x\rt\zetaa,x\Vu\cdot\nablag\zetaa)_2\vert &\leq C(\E^N)(\vert x\zetaa\vert_2^2+\vert \zetaa\vert_2\vert x\zetaa\vert_2 \\
&\leq C(\E^N)E^\alpha
\label{c2:estimation02}\end{aligned}
\end{equation}
where $C$ is continuous and non decreasing.
For the control of the second term of the second line of \eqref{c2:derivee_energiea}, one writes:
\begin{align*}
\frac{1}{\mu}(\G x\psia,x\Vu\cdot\nablag\psia)_2&=\frac{1}{\mu}(\G x\psia,\Vu\cdot(x\psia))_2-\frac{1}{\mu}(\G x\psia,(\Vu\cdot\nablag x)\psia)_2.
\end{align*}
We use Proposition \ref{c2:329} to write (recall the notations of $M$ given by \eqref{c2:notationM}):
\begin{align*}
\vert \frac{1}{\mu}(\G x\psia,\Vu\cdot(x\psia))_2\vert \leq M\vert \Vd\vert_{W^{1,\infty}}\vert\B x\psi\vert_2^2
\end{align*}
and again, using the Proposition \eqref{c2:314} one can control the $W^{1,\infty}$ norm of $\Vd$ by the energy and get 
\begin{equation}\vert\frac{1}{\mu}(\G x\psia,\Vu\cdot(x\psia))_2\vert \leq C(\E^N)E^\alpha.
\label{c2:estimation03}
\end{equation}
We now use Proposition \ref{c2:318} with $s=1$ to compute:
\begin{align*}
\vert \frac{1}{\mu} (\G x\psia,(\Vu\cdot\nablag x)\psia)_2 \vert &\leq \mu M\vert\B x\psia\vert_2\vert (\Vu\cdot\nablag x)\psia\vert_2
\end{align*}
and one can prove, using the definition of $\B$ and standard Sobolev estimates:
\begin{align*}
\vert (\Vu\cdot\nablag x)\psia\vert_2 \leq \vert\Vu\vert_{H^{t_0}}\vert\B\psia\vert_2
\end{align*}
and therefore, using Proposition \eqref{c2:314} again to control $\vert\Vu\vert_{H^{t_0}}$ by the energy, one finally gets 
\begin{equation}
\vert \frac{1}{\mu} (\G x\psia,(\Vu\cdot\nablag x)\psia)_2\vert \leq C(\E^N)E^\alpha.\label{c2:estimation04}
\end{equation}
Putting together \eqref{c2:estimation03} and \eqref{c2:estimation04}, one proved 
\begin{equation}
\vert\frac{1}{\mu}(\G x\psia,x\Vu\cdot\nablag\psia)_2\vert \leq C(\E^N)E^\alpha. \label{c2:estimation05}
\end{equation}
The first term of the third line of \eqref{c2:derivee_energiea} is estimated by using Proposition \ref{c2:318}. The only non trivial remaining term to control in \eqref{c2:derivee_energiea} is the last one, which is the commutator $\frac{1}{\mu}([\G,x]\psia,x\rt\zetaa)_2$. Recall that $\vert\alpha\vert >1$ and that:
 \begin{align*}
\psia = \partial^\alpha\zeta-\epsilon\underline{w}\partial^\alpha\psi
 \end{align*}
 and one gets therefore, using Proposition \ref{c2:commut_x}, one can control both of these terms by $C(\mathcal{E}^N)E^\alpha$. One can obtain by summing on all $\alpha,1\leq \vert\alpha\vert\leq N-1$ the following energy estimate:
$$\frac{d}{dt}\E_x^N \leq C(\E^N)\E^N_x$$ with $C$ a continuous function of its arguments. Using a Gronwall's Lemma, one can conclude and end the proof of the Theorem.\qquad$\Box$.

\end{proof}

\begin{appendix}
\section{Estimates on the flat strip $\S$}
We recall the notation $a\vee b = \max(a,b)$ and we define $L^\infty H^s = L^\infty((-1;0);H^s(\R^d))$ and use the notation $\langle a\rangle_{s>t_0} = a$ if $s> t_0$ and else $0$ .
\begin{proposition}\label{c2:B5} Let $t_0>d/2$. If $s\geq -t_0$, $f,g\in L^\infty H^{t_0}\cap H^{s,0}$, one has $fg\in H^{s,0}$ and $$\norme{fg}_{H^{s,0}} \leq C\norme{f}_{L^\infty H^{t_0}}\norme{g}_{H^{s,0}} + \langle\norme{f}_{H^{s,0}}\norme{g}_{L^\infty H^{t_0}}\rangle_{s>t_0}.$$
\end{proposition}
The following Proposition states a $L^\infty$ embedding result for the Beppo-Levi spaces:
\begin{proposition}\label{c2:212}
For all $s\in\R$: \begin{enumerate}[label=(\arabic*),itemsep=0pt]
\item The mapping $u\mapsto u_{\vert z=0}$ extends continuously from $H^{s+1,1}$ to $H^{s+1/2}(\R^d)$.
\item The space $H^{s+1/2,1}$ is continuously embedded in $L^\infty H^s$.

\end{enumerate}
\end{proposition}

\section{The Dirichlet Neumann Operator}\label{c2:appendixA}
Here are for the sake of convenience some technical results about the Dirichlet Neumann operator, and its estimates in Sobolev norms. See \cite{david} Chapter 3 for complete proofs. The first two  propositions give a control of the Dirichlet-Neumann operator. 

\begin{proposition}\label{c2:314}
Let $t_0$>d/2, $0\leq s \leq t_0+3/2$ and $(\zeta,\beta)\in H^{t_0+1}\cap H^{s+1/2}(\mathbb{R}^d)$ such that  \begin{equation*} \exists h_0>0,\forall X\in\mathbb{R}^d,  \epsilon\zeta(X)-\beta b(X) +1 \geq h_0.\end{equation*} 
\begin{enumerate}[label = ( \arabic*)]
\item\quad The operator $G$ maps continuously $\overset{.}H{}^{s+1/2}(\mathbb{R}^d)$ into $H{}^{s-1/2}(\mathbb{R}^d)$ and one has 
\begin{equation*}
\vert G\psi\vert_{H^{s-1/2}} \leq \mu^{3/4} M(s+1/2) \vert\mathfrak{P}\psi\vert_{H^s},
\end{equation*}
where $M(s+1/2)$ is a constant of the form $C(\frac{1}{h_0},\vert\zeta\vert_{H^{t_0+1}},\vert b\vert_{H^{t_0+1}},\vert\zeta\vert_{H^{s+1/2}},\vert b\vert_{H^{s+1/2}})$.

\item \quad The operator $G$ maps continuously $\overset{.}H{}^{s+1}(\mathbb{R}^d)$ into $H{}^{s-1/2}(\mathbb{R}^d)$ and one has 
\begin{equation*}
\vert G\psi\vert_{H^{s-1/2}} \leq \mu M(s+1) \vert\mathfrak{P}\psi\vert_{H^{s+1/2}},
\end{equation*}
where $M(s+1)$ is a constant of the form $C(\frac{1}{h_0},\vert\zeta\vert_{H^{t_0+1}},\vert b\vert_{H^{t_0+1}},\vert\zeta\vert_{H^{s+1}},\vert b\vert_{H^{s+1}})$.
\end{enumerate}

Moreover, it is possible to replace $G\psi$ by $\underline{w}$ in the previous result, where $\underline{w} = \frac{G\psi+\epsilon\mu\nablag\zeta\cdot\nablag\psi}{1+\epsilon^2\mu\vert\nablag\zeta\vert^2}$(vertical component of the velocity $U=\nabla_{X,z}\Phi$ at the surface).
\end{proposition}
\begin{remark} In all this paper, we consider the Water-Waves problem in finite depth. This is crucial for all these regularity results on $G$. For instance, in the linear case $\zeta=b =0$, the Dirichlet-Neumann operator is $\D\tanh(\D)$ in finite depth, while it is $\D$ in infinite depth. The low frequencies are therefore affected differently. 
\end{remark}
\begin{proposition}\label{c2:318}
Let $t_0>d/2$, and $0\leq s\leq t_0+1/2$. Let also $\zeta,b\in H^{t_0+1}(\mathbb{R}^d)$ be such that  $$\exists h_0>0, \forall X\in\mathbb{R}^d, 1+\epsilon\zeta(X)-\beta b(X) \geq h_0.$$ Then, for all $\psi_1$, $\psi_2\in \overset{.}H{}^{s+1/2}(\mathbb{R}^d)$, we have  $$(\Lambda^sG\psi_1,\Lambda^s\psi_2)_2 \leq\mu M_0 \vert \mathfrak{P}\psi_1\vert_{H^s}\vert \mathfrak{P}\psi_2\vert_{H^s},$$
where $M_0$ is a constant of the form $C(\frac{1}{h_0},\vert\zeta\vert_{H^{t_0+1}},\vert b\vert_{H^{t_0+1}})$.
\end{proposition}

The second result gives a control of the shape derivatives of the Dirichlet-Neumann operator. More precisely, we define the  open set  $\mathbf{\Gamma}\subset H^{t_0+1}(\mathbb{R}^d)^2$ as:
$$\mathbf{\Gamma} =\lbrace \Gamma=(\zeta,b)\in H^{t_0+1}(\mathbb{R}^d)^2,\quad \exists h_0>0,\forall X\in\mathbb{R}^d, \epsilon\zeta(X) +1-\beta b(X) \geq h_0\rbrace$$ and, given a $\psi\in \overset{.}H{}^{s+1/2}(\mathbb{R}^d)$, the mapping: \begin{equation}\label{c2:mapping}G[\epsilon\cdot,\beta\cdot] : \left. \begin{array}{rcl}
&\mathbf{\Gamma} &\longrightarrow H^{s-1/2}(\mathbb{R}^d) \\
&\Gamma=(\zeta,b) &\longmapsto G[\epsilon\zeta,\beta b]\psi.
\end{array}\right.\end{equation} We can prove the differentiability of this mapping. The following Proposition gives estimates of the shape derivatives of $\G$.
%
%

\begin{proposition}\label{c2:328b}
Let $t_0>d/2$ and $(\zeta,b)\in H^{t_0+1}$ be such that \begin{equation*} \exists h_0>0,\forall X\in\mathbb{R}^d,  \epsilon\zeta(X)-\beta b(X) +1 \geq h_0.\end{equation*} Then, for all $0\leq s\leq t_0+1/2$, $$\vert d^j G(h,k)\psi\vert_{H^{s-1/2}} \leq M_0 \mu^{3/4} \prod_{m=1}^j \vert (\epsilon h_m,\beta k_m)\vert_{H^{t_0+1}} \vert \B\psi\vert_{H^s}.$$
\end{proposition}
The following commutator estimate is useful (see \cite{david} Proposition 3.30):

\begin{proposition}\label{c2:329}
Let $t_0>d/2$ and $\zeta, b \in H^{t_0+2}(\mathbb{R}^d)$ such that:  \begin{equation*} \exists h_0>0,\forall X\in\mathbb{R}^d,  \epsilon\zeta(X)-\beta b(X) +1 \geq h_0.\end{equation*} 
For all $\underline{V}\in H^{t_0+1}(\mathbb{R}^d)^2$ and $u\in H^{1/2}(\mathbb{R}^d)$, one has 

\begin{equation*}
((\underline{V}\cdot\nabla^{\gamma} u),\frac{1}{\mu}Gu)\leq M\vert\underline{V}\vert_{W^{1,\infty}}\vert \mathfrak{P} u\vert_2^2,
\end{equation*}
where $M$ is a constant of the form $C(\frac{1}{h_0},\vert\zeta\vert_{H^{t_0+2}},\vert b\vert_{H^{t_0+2}})$.
\end{proposition}

\end{appendix}

The author has been partially funded by the ANR project Dyficolti ANR-13-BS01-0003-01.
\bibliographystyle{plain}

\bibliography{estimation_dispersion}

\end{document}